\documentclass[a4paper,12pt]{article}
\usepackage[top=2.5cm,bottom=2.5cm,left=2.5cm,right=2.5cm]{geometry}
\usepackage{cite, amsmath, amssymb,amsfonts,mathrsfs,amsthm}
\usepackage{latexsym}
\usepackage{graphicx,booktabs,multirow}
\usepackage{tikz}
\usetikzlibrary{decorations.pathreplacing}
\usetikzlibrary{intersections}
\usepackage{enumerate}
\usepackage{appendix}
\usepackage{graphicx}
\usepackage{subfigure}
\newtheorem{theorem}{Theorem}[section]
\newtheorem{proposition}[theorem]{Proposition}
\newtheorem{lemma}[theorem]{Lemma}
\newtheorem{corollary}[theorem]{Corollary}
\newtheorem{remark}[theorem]{Remark}

\newtheorem{question}[theorem]{Question}
\newtheorem{definition}[theorem]{Definition}
\allowdisplaybreaks[4]

\usepackage[section]{placeins}
\def\NAT@def@citea{\def\@citea{\NAT@separator}}
\allowdisplaybreaks

\begin{document}
\vspace*{10mm}

\noindent
{\Large \bf On determinants of tournaments and $\mathcal{D}_k$}

\vspace*{7mm}

\noindent
{\large \bf Jing Zeng, Lihua You$^{*}$}
\noindent

\vspace{7mm}

\noindent
School of Mathematical Sciences, South China Normal University,  Guangzhou, 510631, P. R. China,
e-mail: {\tt 2022021944@m.scnu.edu.cn},\quad{\tt 20041023@m.scnu.edu.cn}. \\[2mm]
$^*$ Corresponding author
\noindent

\vspace{7mm}

\noindent
{\bf Abstract} \
\noindent
Let $T$ be a tournament with $n$ vertices $v_1,\ldots,v_n$. The skew-adjacency matrix of $T$ is the $n\times n$ zero-diagonal matrix $S_T = [s_{ij}]$ in which $s_{ij}=-s_{ji}=1$ if $ v_i $ dominates $ v_j $. We define the determinant $\det(T)$ of $ T $ as the determinant of $ S_T $. It is well-known that $\det(T)=0$ if $n$ is odd and $\det(T)$ is the square of an odd integer if $n$ is even. Let $\mathcal{D}_k$ be the set of tournaments whose all subtournaments have determinant at most $ k^{2} $, where $k$ is a positive odd integer. The necessary and sufficient condition for $T\in \mathcal{D}_1$ or $T\in \mathcal{D}_3$ has been characterized in $2023$. In this paper, we characterize the set $\mathcal{D}_5$, obtain some properties of $\mathcal{D}_k$. Moreover, for any positive odd integer $k$, we give a construction of a tournament $T$ satisfying that $\det(T)=k^2$, and $T\in \mathcal{D}_k\backslash\mathcal{D}_{k-2}$ if $k\geq 3$, which implies $\mathcal{D}_k\backslash\mathcal{D}_{k-2}$ is not an empty set for $k\geq 3$.
 \\[2mm]


\noindent
{\bf AMS:} \ 05C20, 05C50

\baselineskip=0.30in

\section{Introduction}

\hspace{1.5em} A \textit{tournament} is a directed graph with exactly one arc between each pair of vertices. If the arc between two vertices $ u $ and $ v $ is directed  from $ u $ to $ v $, we say that $ u $ dominates $ v $, and write $u\rightarrow v$. Throughout this paper, we mean by an \textit{$n$-tournament}, a tournament with $ n $ vertices. Let $T$ be an $n$-tournament with vertex set $\{v_1,\ldots,v_n\}$. The \textit{skew-adjacency matrix} of $T$ is the $n\times n$ zero-diagonal matrix $S_T = [s_{ij}]$ in which $s_{ij}=-s_{ji}=1$ if $ v_i $ dominates $ v_j $. We define the \textit{determinant} $\det(T)$ of $T$ as the determinant of $S_T$. It is well-known that the determinant of a skew-symmetric matrix with even order is the square of its Pfaffian by Cayley\cite{PFAFFIAN}. By using Pfaffian polynomial, Fisher and Ryan\cite{DET} proved that $\det(T)=0$ if $n$ is odd and $\det(T)$ is the square of an odd integer if $n$ is even. 

Let $ T $ be an $n$-tournament. We denote by $V(T)$ the vertex set of $T$. For $ X \subseteq V(T) $, we denote by $ T[X] $ the subtournament of $ T $ induced by $ X $. To simplify, for $\{v_1,\ldots,v_k\}\subseteq V(T)$, we denote by $ T[v_1,v_2,\ldots,v_k]$ the subtournament $T[\{v_1,v_2,\ldots,v_k\}]$ of $T$. We say that $T$ contains a tournament $H$ if $H$ is isomorphic to a subtournament of $T$. Let $X$ and $Y$ be subsets of $V(T)$ such that $X\cap Y=\emptyset$. If $v\rightarrow u$ for any $v\in X$ and any $u\in Y$, we write $X\rightarrow Y$, and if $X=\{v\}$, we write $v\rightarrow Y$. 

The \textit{join} of a tournament $T_1$ to a tournament $T_2$, denoted by $T_1\rightarrow T_2$, is the tournament obtained from $T_1$ and $T_2$ by adding an arc from each vertex of $T_1$ to all vertices of $T_2$. It is easy to see that $V(T_1)\rightarrow V(T_2)$ in $T_1\rightarrow T_2$. Let $K_1$ be a $1$-tournament with vertex $u$. The join of a tournament $T$ to $K_1$ is denoted by $T\rightarrow u$, $T^{+}$\cite{UNI} for short, and the join of $K_1$ to a tournament $T$ is denoted by $u \rightarrow T$, $T^{-}$ for short. 

A tournament $ T $ is \textit{transitive} if there is no 3-cycle in $ T $, or equivalently, if it is possible to order its vertices into $v_1,\ldots,v_n$ such that $v_i\rightarrow v_j$ if and only if $i<j$. The determinant of a transitive tournament is equal to $0$ or $1$\cite{DETTRANSI}.

\begin{definition}\label{DefBLWOUP}
Let $R$ be an $n$-tournament with vertices $v_1,\cdots,v_n$, $T_1,\cdots,T_n$ be tournaments. A tournament $R(T_1,\cdots,T_n)$ is obtained by replacing each vertex $v_i$ with the tournament $T_i$ for each $1\leq i\leq n$, and adding arcs between $V(T_i)$ and $V(T_j)$ such that $V(T_i)\rightarrow V(T_j)$ if $v_i\rightarrow v_j$ for $1\leq i,j \leq n$, we call such $R(T_1,\ldots,T_n)$ is a blowup of $R$ by $T_1,\ldots,T_n$.	
\end{definition}

 Follow the notation in Definition \ref{DefBLWOUP}, if $T_i$ is transitive for each $1\leq i\leq n$ and $a_i$ = $|V(T_i)|$, we call $R(T_1,\cdots,T_n)$ the $transitive$ $(a_1,\cdots,a_n)$-$blowup$ of $R$ (transitive blowup of $R$ for short), denoted by $R(a_1,\cdots,a_n)$\cite{BLOWUP}.

The \textit{switch} of a tournament $ T $, with respect to a subset $ W $ of $ V $, is the tournament obtained by reversing all the arcs between $ W $ and $V\backslash W$ (If $W=\emptyset$ or $W=V$, then the switch of $T$ is $T$ itself). If $T'$ is a switch of $T$, we say $T'$ and $T$ are switching equivalent. It is well-known that two tournaments $T_1$ and $T_2$ with the same vertex set are switching equivalent if and only if their skew-adjacency matrices are $\{\pm1\}$-diagonally similar\cite{TWOGRAPH}. Hence, the skew-adjacency matrices of $T_1$ and $T_2$ have the same principal minors. Then we have the following result.

\begin{proposition}\label{Minors}
Let $T_1$ and $T_2$ are two tournaments with the same vertex set $V$ such that $T_1$ and $T_2$ are switching equivalent. Then $T_1[U]$ is switching equivalent to $T_2[U]$, and $\det(T_1[U])=\det(T_2[U])$ for any non-empty subset $U\subseteq V$.
\end{proposition}

For every positive odd integer $k$, let $\mathcal{D}_k$ be the set of tournaments whose all subtournaments have determinant at most $ k^{2} $, or equivalently, the principal minors of their skew-adjacency matrices do not exceed $k^2$\cite{DTHREE}.

By Proposition \ref{Minors}, if two tournaments $T_1$ and $T_2$ are switching equivalent, then $T_1\in \mathcal{D}_k$ if and only if $T_2\in \mathcal{D}_k$. Thus $\mathcal{D}_k$ is closed under the switching operation.

A \textit{diamond} is a $4$-tournament consisting of a vertex dominating or dominated by a 3-cycle. In this paper, we use $D$ to denote a diamond. The determinant of a diamond is $9$, and the determinant of a $4$-tournament is 9 if it is a diamond and 1 otherwise\cite{DIACHA}. A tournament contains no diamonds if and only if it is switching equivalent to a transitive tournament\cite{TRANSI}. Note that the subtournaments of a transitive tournament are also transitive tournaments, and the determinant of a transitive tournament is $1$ or $0$, then we have the following Proposition \ref{Done}.

\begin{proposition}\label{Done}
	Let $T$ be a tournament. Then the following assertions are equivalent:
	\item{\rm(i)} $T\in \mathcal{D}_1$.
	\item{\rm(ii)} $T$ is switching equivalent to a transitive tournament.
	\item{\rm(iii)} $T$ contains no diamonds. 
\end{proposition}

The authors \cite{DTHREE} characterized $\mathcal{D}_3$ as follows.

\begin{theorem}{\rm(\!\!\text{ Boussaïri et al.}\cite{DTHREE})}\label{Dthree}
	Let $T$ be a tournament. Then the following assertions are equivalent:
	\item{\rm(i)} $T\in \mathcal{D}_3$.
	\item{\rm(ii)} $T$ is switching equivalent to a transitive tournament or a transitive blowup of a diamond.
	\item{\rm(iii)} All the $6$-subtournaments of $T$ are in $\mathcal{D}_3$.
\end{theorem}

By Proposition \ref{Done} and Theorem \ref{Dthree}, a tournament $T\in \mathcal{D}_3\backslash \mathcal{D}_1$ if and only if $T$ is switching equivalent to a transitive blowup of a diamond.

Clearly, $\mathcal{D}_k=(\mathcal{D}_k\backslash \mathcal{D}_{k-2}) \cup (\mathcal{D}_{k-2}\backslash \mathcal{D}_{k-4})\cup \cdots \cup (\mathcal{D}_3\backslash \mathcal{D}_1) \cup \mathcal{D}_1$ for odd $k$. If $\mathcal{D}_j\backslash \mathcal{D}_{j-2}\text{ }(j=3,5,\ldots,k)$ have been characterized, then $\mathcal{D}_k$ can be characterized.   

Based on the above results, a natural question is to characterize $\mathcal{D}_k$ for odd $k\geq 5$, and another question worth considering is whether there exists a tournament $T$ such that $\det(T)=k^2$ for any odd $k\in N$, furthermore, whether $\mathcal{D}_k\backslash \mathcal{D}_{k-2}$ is not an empty set.

In this paper, we characterize the set $\mathcal{D}_5$, explore the properties of $\mathcal{D}_k$. Moreover, for any positive odd $k$, we give a construction of a tournament $T$ satisfying that $\det(T)=k^2$ and $T\in \mathcal{D}_k\backslash\mathcal{D}_{k-2}$ if $k\geq 3$. By using this construction, we prove that there exists a tournament whose determinant is $k^2$ for any odd $k$, which solve a question in \cite{UNI}, and $\mathcal{D}_k\backslash \mathcal{D}_{k-2}$ is not an empty set for odd $k\geq 3$. 

\section{Preliminaries}\label{sec-pre}
\hspace{1.5em}In this section, we introduce the key tool in our study, further consider another invariant between tournament $T$ and its switches, and give a construction of tournaments. Moreover, we investigate the relation between determinants and the number of diamonds in $6$-tournaments, and obtain some necessary and sufficient conditions for determining whether the determinant of a $6$-tournament is 25.

Theorem \ref{Djoin} \cite{UNI} is a key tool in our study. We give another proof of (ii) below.

\begin{theorem}{\rm(\!\!\text{ Belkouche et al.}\cite{UNI})}\label{Djoin}
	Let $T_1$ and $T_2$ be tournaments with $p$ and $q$ vertices respectively.
	
	{\rm \item(i)} If $p$ and $q$ are even, then $\det(T_1 \rightarrow T_2)=\det(T_1) \cdot \det(T_2)$.
	
	{\rm \item(ii)} If $p$ and $q$ are odd, then $\det(T_1 \rightarrow T_2)=\det(T_1^{+}) \cdot \det(T_2^{+})$.
\end{theorem}

\begin{proof}[Proof of {\rm (ii)}]
	Since $T^{+}$ is switching equivalent to $T^{-}$, we have $\det(T^{-})$=$\det(T^{+})$. If $p$ and $q$ are odd, then we have $\det(T_1^{-}\rightarrow T_2^{+})=\det(T_1^{-})\cdot\det(T_2^{+})$ by (i). 
	
	Let $T_3$ be a tournament with two vertices $u_1,u_2$ such that $u_1\rightarrow u_2$. Then $(T_1\rightarrow T_2)\rightarrow T_3$ is switching equivalent to $(u_2\rightarrow T_1)\rightarrow (T_2\rightarrow u_1)$ with respect to the subset $W=\{u_2\}$ of $V=V(T_1)\cup V(T_2)\cup V(T_3)$. Since $p$ and $q$ are odd, then $p+q$ is even. By $|V(T_3)|=2$, $\det(T_3)=1$ and (i), we have	
	
	$$
	\begin{aligned}
	\det(T_1\rightarrow T_2)&=\det((T_1\rightarrow T_2)\rightarrow T_3)\\
	&=\det((u_2\rightarrow T_1)\rightarrow (T_2\rightarrow u_1))\\
	&=\det(u_2\rightarrow T_1)\cdot \det(T_2\rightarrow u_1)\\
	&=\det(T_1^{-})\cdot\det(T_2^{+})\\
	&=\det(T_1^{+})\cdot\det(T_2^{+}),\\
	\end{aligned}
	$$
	and we complete the proof.
\end{proof}

In \cite{DETTRANSI}, the authors showed that the determinant of a transitive tournament with even order is $1$. As an application of Theorem \ref{Djoin}, we give another proof of Proposition \ref{Dettransi}.

\begin{proposition}{\rm(\!\!\text{ Deng et al.}\cite{DETTRANSI})}\label{Dettransi}
	Let $n$ be even, $T$ be a transitive $n$-tournament. Then $\det(T)=1$. 
\end{proposition}

\begin{proof}
	If $n=2$, then $\det(T)=1$. If $n\geq 4$, since $T$ is transitive, there exists an ordering of vertices $v_1,\ldots,v_n$ of $T$ such that $v_i\rightarrow v_j$ for $i<j$. Then $T[v_1,\ldots,v_n]=T[v_1,\ldots,v_{n-2}]\rightarrow T[v_{n-1},v_{n}]$.
	
	By (i) of Theorem \ref{Djoin} and $\det(T[v_{n-1},v_n])=1$, we have $\det(T[v_1,\ldots,v_n])=\det(T[v_1,\\\ldots,v_{n-2}])\cdot \det(T[v_{n-1},v_{n}])= \det(T[v_1,\ldots,v_{n-2}])$.
	
	By repeatedly using (i) of Theorem \ref{Djoin}, we have $\det(T[v_1,\ldots,v_n])=\det(T[v_1,\ldots,v_{n-2}\\])=\cdots=\det(T[v_1,v_2])=1$.
\end{proof}

We denote by $\delta_T$ the number of diamonds of $T$ where $T$ is a tournament. The following lemma shows that the number of diamonds of a tournament $T$ is an invariant under switching operation.

\begin{lemma}\label{diamond}
	Let $T_1$ and $T_2$ be $n$-tournaments with the same vertex set $V$ such that $T_1$ and $T_2$ are switching equivalent, $X=\{v_1,v_2,v_3,v_4\}\subseteq V$.
	{\rm \item(i)} $T_1[X]$ is a diamond if and only if $T_2[X]$ is a diamond.
	{\rm \item(ii)} $\delta_{T_1}=\delta_{T_2}$. 
\end{lemma}

\begin{proof}
	If $T_i[X]$ is a diamond for $i\in \{1,2\}$, then $\det(T_i[X])=9$. By Proposition \ref{Minors}, we have $\det(T_j[X])=9$ for $j\in \{1,2\}\backslash\{i\}$, then $T_j[X]$ is also a diamond. Thus (i) holds.
	
	If there is a diamond with vertices $v_{i1},v_{i2},v_{i3},v_{i4}$ in $T_i$ for $i\in \{1,2\}$, then $T_j[v_{i1},v_{i2},v_{i3},\\v_{i4}]$ is a diamond in $T_j$ for $j\in \{1,2\}\backslash \{i\}$ by (i), and thus $\delta_{T_j}\geq \delta_{T_i}$. Combining the cases of $i=1$ and $i=2$, we have $\delta_{T_1}=\delta_{T_2}$, (ii) holds.
\end{proof}

The following lemma is crucial to explore the properties of $6$-tournaments.

\begin{lemma}\label{fzt}
  Let $ T $ be a $5$-tournament. Then $ \delta_T\in \{0,2\}$.
\end{lemma}
\begin{proof}
	Let $V(T)=\{v_1,v_2,v_3,v_4,v_5\}$. If $\delta_T > 0$, we assume that $ T[v_1,v_2,v_3,v_4] $ is a diamond such that $ v_2\rightarrow v_3\rightarrow v_4\rightarrow v_2 $, $ v_1\rightarrow\{v_2,v_3,v_4\} $ or $ v_1\leftarrow\{v_2,v_3,v_4\}. $
	
	Let $W=\{v_i\in V(T)\hspace{0.15cm}|\hspace{0.15cm} v_i\rightarrow v_5\}$, $ T' $ be a switch of $ T $ with respect to $W$. Then $T'$ satisfies that $v_5\rightarrow \{v_1,v_2,v_3,v_4\}$ in $ T' $, and $ T'[v_1,v_2,v_3,v_4] $ is also a diamond by Lemma \ref{diamond}. Assume that $ \{v'_1,v'_2,v'_3,v'_4\} = \{v_1,v_2,v_3,v_4\}$ such that $ v'_2\rightarrow v'_3\rightarrow v'_4\rightarrow v'_2 $, $ v'_1\rightarrow \{v'_2,v'_3,v'_4\}$ or $ v'_1\leftarrow \{v'_2,v'_3,v'_4\}$ in $ T' $, then there exists exactly one $3$-cycle $T'[v_2',v_3',v_4']$ in $T'$ since a $3$-cycle in $T'$ does not contain $v_5$ and $v_1'$. Consequently, there are exactly two diamonds $ T'[v'_1,v'_2,v'_3,v'_4] $ and $ T'[v_5,v'_2,v'_3,v'_4] $ in $ T' $. Hence there are exactly two diamonds in $ T $ by $ \delta_T = \delta_{T'}=2 $, and thus $\delta_T\in \{0,2\}$.
\end{proof}

There are many research on the number of diamonds of a tournament. For example, Bouchaala\cite{DIANINE} proved that $\delta_T$ is either $0$, $n-3$, $2n-8$ or at least $2n-6$ for an $n$-tournament $T(n\geq 9)$, Gunderson and Semeraro\cite{DIABOUND} proved that an $r$-uniform hypergraph in which every ($r+1)$-subset contains at most $2$ hyperedges has at most $\frac{n}{r^2} \tbinom{n}{r-1}$ hyperedges, in particular, an $n$-tournament contains at most $\frac{n}{16} \tbinom{n}{3}$ diamonds, Belkouche et al. \cite{DIACHA} characterized the $n$-tournaments with the maximum number of diamonds when $n\equiv 0\text{ }(\text{mod }4)$ and $n\equiv 3\text{ }(\text{mod }4)$ in the case of the existence of skew-conference matrices, and so on. By using Lemma \ref{fzt}, we obtain some bounds on $\delta_T$.

\begin{lemma}\label{bounddia}
	Let $n\geq 5$, $ T $ be an $n$-tournament. Then $ \delta_T=0$, or $ n-3\leq\delta_T\leq \frac{2}{5}\tbinom{n}{4}$.
\end{lemma}

\begin{proof}
	If $\delta_T\neq0$, let $V(T)=\{v_1,v_2,\ldots,v_n\}$ such that $T[v_1,v_2,v_3,v_4]$ is a diamond. For $5\leq i \leq n$, $T[v_1,v_2,v_3,v_4,v_i]$ contains two diamonds by Lemma \ref{fzt}, one is $T[v_1,v_2,v_3,v_4]$, and the other contains vertex $v_i$, denoted by $Q_i$. Since $Q_i$ contains vertex $v_i$ and $Q_j$ contains vertex $v_j\text{ }(i\neq j)$, $Q_i$ and $Q_j$ are not the same diamond in $T$. Then $T$ contains diamonds $T[v_1,v_2,v_3,v_4]$ and $Q_5,\ldots,Q_n$. Therefore, $\delta_T\geq 1+(n-4)=n-3$.
	
	There are $\tbinom{n}{5}$ $5$-subtournaments in $T$, and each $5$-subtournament in $T$ contains exactly $0$ or $2$ diamonds by Lemma \ref{fzt}, then these $5$-subtournaments in total contain at most $2\tbinom{n}{5}$ diamonds (including repeated diamonds). Moreover, each diamond of $T$ is contained exactly in $n-4$ $5$-subtournaments of $T$. Therefore, $(n-4)\cdot \delta_T\leq 2\tbinom{n}{5}$, and we have $\delta_T\leq \frac{2}{5}\tbinom{n}{4}$.
\end{proof}

By Lemma \ref{bounddia}, we have the following corollary immediately, and then we can study the properties of $6$-tournaments.

\begin{corollary}\label{rebounddia}
	Let $T$ be a $6$-tournament. Then $\delta_T=0$, or $3\leq \delta_T\leq 6$.
\end{corollary}

\begin{proposition}\label{sixtran}
  Let $ T $ be a 6-tournament. Then 
  \item{\rm(i)} there exists a $6$-tournament $T'$ switching equivalent to $T$ such that $T'$ contains a transitive $4$-subtournament;
  \item{\rm(ii)} furthermore, if $\delta_T<6$, then there exists a $6$-tournament $T''$ switching equivalent to $T$ such that $T''$ contains a transitive $5$-subtournament.
\end{proposition}

\begin{proof}
	By Corollary \ref{rebounddia}, we have $\delta_T\leq 6$, but $T$ has $\tbinom{6}{4}=15$ $4$-subtournaments, then there exists a $4$-subtournament $T_1$ of $T$ such that $T_1$ is not a diamond. By Proposition \ref{Done}, $T_1$ is switching equivalent to a transitive tournament $T_2$ with respect to a subset $W$ of $V(T_1)$, and thus there exists a $6$-tournament $T'$ switching equivalent to $T$ such that $T'$ contains the transitive $4$-subtournament $T_2$ with respect to $W$ as $W\subseteq V(T_1) \subseteq V(T)$. Then (i) holds.
	
	Furthermore, if $\delta_T<6$, then there exists at least one $5$-subtournament $T_3$ of $T$ such that $T_3$ contains no diamonds by the proof of Lemma \ref{bounddia}, and thus there exists a $6$-tournament $T''$ switching equivalent to $T$ such that $T''$ contains a transitive $5$-subtournament by Proposition \ref{Done} and similar proof of (i). Therefore, (ii) holds.
\end{proof}

Before giving the next property of $6$-tournaments, we introduce some definitions and notation.

\begin{definition}\label{defpsi}
Let $T$ be an $n$-tournament, $X$ be a subset of $V(T)$ such that $T[X]$ is transitive and $|X|=k$. For any $u\in V(T)\backslash X$ and the ordering of $X$, $\{v_1,\ldots,v_k\}$, which satisfies $v_1\rightarrow v_2\rightarrow \cdots \rightarrow v_k$ in $T$, we define the dominating relation between $u$ and $X$ by $\psi_T(u,X)=(\alpha_1,\ldots,\alpha_t)$, where nonzero integers $\alpha_1,\ldots,\alpha_t$ and a partition $X(i,\alpha_i)(i=1,\cdots,t)$ of $X$ satisfy that $|\alpha_1|+\cdots+|\alpha_t|=k$, $\alpha_i\alpha_{i+1}<0$ for $1\leq i\leq t-1$, $X(1,\alpha_1)=\{v_1,\ldots,v_{|\alpha_1|}\}$, $X(j  ,\alpha_j)=\{v_{|\alpha_1|+\cdots+|\alpha_{j-1}|+1},\ldots,v_{|\alpha_1|+\cdots+|\alpha_j|}\}$ for $2\leq j\leq t$, and the arcs between $u$ and $X$ satisfy that $u\rightarrow X(i,\alpha_i)$ if $\alpha_i>0$, and $u\leftarrow X(i,\alpha_i)$ if $\alpha_i<0$.
\end{definition}

In Definition \ref{defpsi}, by the definition of $\alpha_i$, it is easy to see that $1\leq t\leq n-1$, $u\leftarrow X(i+1,\alpha_{i+1})$ if $u\rightarrow X(i,\alpha_i)$, and $u\rightarrow X(i+1,\alpha_{i+1})$ if $u\leftarrow X(i,\alpha_i)$. Now we give a construction of tournaments.

Let $ L_n $ be an $n$-tournament with $V(L_n)=\{u_1,\ldots,u_n\}$ such that $L_n[X_u]$ is transitive with $u_1\rightarrow u_2\rightarrow \cdots \rightarrow u_{n-1}$, and $\psi_{L_n}(u_n,X_u)=((-1)^0,(-1)^1,\ldots,(-1)^{n-2})$, where $X_u=\{u_1,\ldots,u_{n-1}\}$. If an $n$-tournament $T$ is isomorphic to $L_n$, we say $T$ is $L_n$.

\begin{figure}[h]
	\centering
	\includegraphics[scale=0.25]{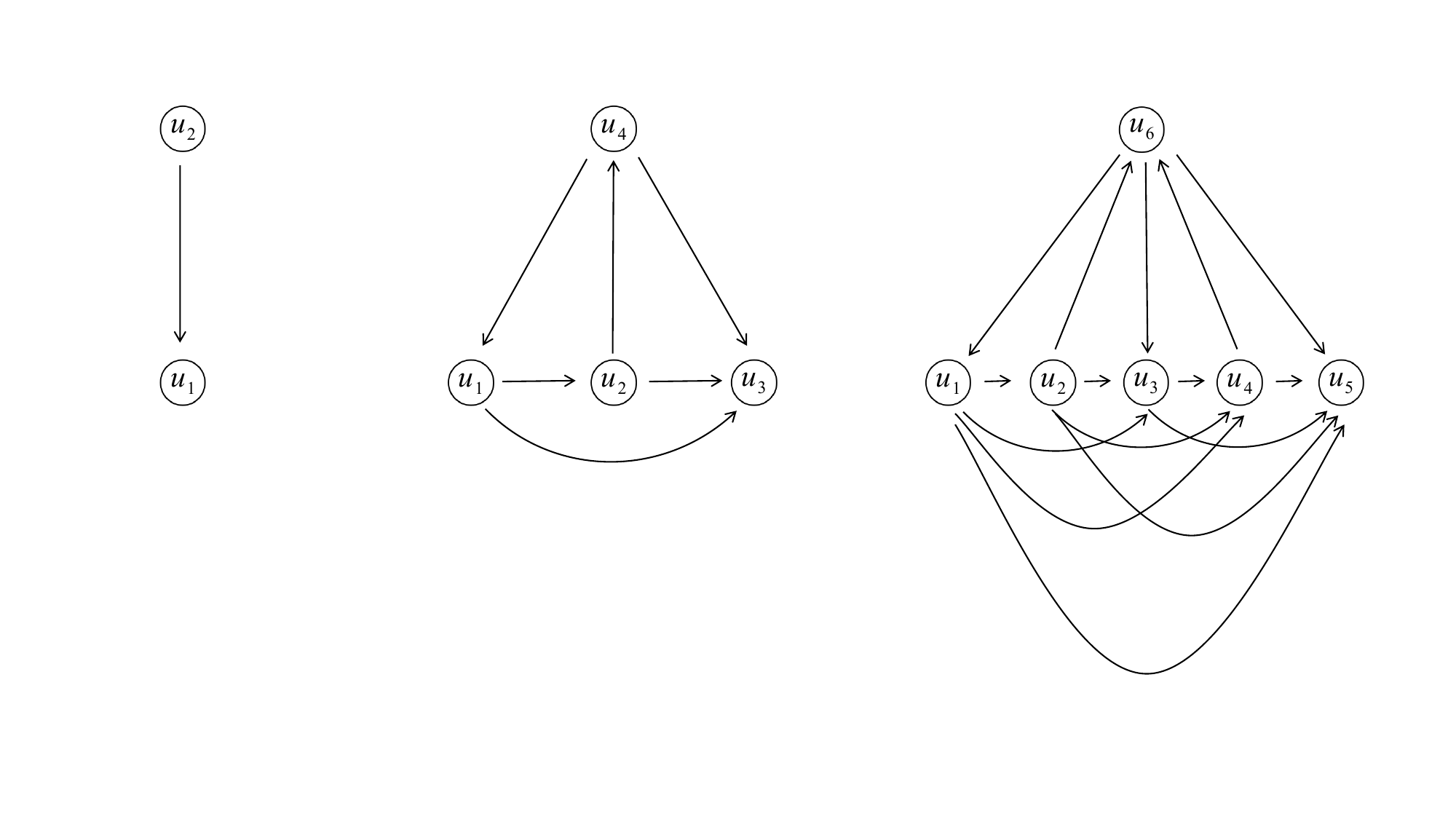}
	\caption{$L_2$, $L_4$ and $L_6$}
\end{figure}

Clearly, $L_2$ is a transitive tournament, $L_4$ is a diamond. In this paper, we use $L_n(a_1,\ldots,a_n)$ to denote a transitive blowup of $L_n$ where $u_i$ is replaced by a transitive tournament with order $a_i$ for each $i\in \{1,\ldots,n\}$. The structures of $L_2$, $L_4$ and $L_6$ are shown in Figure $1$.

\begin{lemma}\label{diainL}
	Let $T$ be an $n$-tournament with $n$ vertices $v_1,\ldots,v_{n-1},v_n$ such that $T[X]$ is transitive with $v_1\rightarrow \cdots \rightarrow v_{n-1}$, $\psi_T(v_n,X)=(\alpha_1,\alpha_2,\ldots,\alpha_t)$ where $X=\{v_1,\ldots,v_{n-1}\}$. If $D$ is a diamond of $T$, then $v_n\in V(D)$, and there exist $i_1,i_2,i_3$ such that $1\leq i_1<i_2<i_3\leq t$, $|V(T)\cap X(i_j,\alpha_{i_j})|=1$ for $1\leq j\leq 3$, $\alpha_{i_1}\alpha_{i_2}<0$ and $\alpha_{i_2}\alpha_{i_3}<0$.
\end{lemma}

\begin{proof}
	If $v_n\notin V(D)$, then $V(D)\subseteq X$. Since $T[X]$ is transitive, $D$ is not a diamond, a contradiction. Therefore, $v_n\in V(D)$.
	
	Now, clearly, $|V(D)\cap X|=3$. If there exists $\alpha_k$ such that $|V(D)\cap X(k,\alpha_k)|=3$, it is easy to see that $D$ is transitive, a contradiction. 
	
	If there exists $\alpha_{k_1}$ such that $|V(D)\cap X(k_1,\alpha_{k_1})|=2$, then there exists another $\alpha_{k_2}$ such that $|V(D)\cap X(k_2,\alpha_{k_2})|=1$. If $\alpha_{k_1}\alpha_{k_2}>0$, it is easy to see that $D$ is transitive, a contradiction; If $\alpha_{k_1}\alpha_{k_2}<0$, then $D$ is switching equivalent to a transitive tournament with respect to $W=V(D)\cap X(k_2,\alpha_{k_2})$, hence $D$ is not a diamond, a contradiction.
	
	Hence, there exist $i_1,i_2,i_3$ such that $i_1<i_2<i_3$, $|V(D)\cap X(i_j,\alpha_{i_j})|=1$ for $j=1,2,3$. If $\alpha_{i_1}\alpha_{i_2}>0$ or $\alpha_{i_2}\alpha_{i_3}>0$, similarly, $D$ is not a diamond, a contradiction. Therefore, $\alpha_{i_1}\alpha_{i_2}<0$ and $\alpha_{i_1}\alpha_{i_2}<0$. We complete the proof.
\end{proof}

The following proposition show that the relation between determinants and the number of diamonds in $6$-tournaments, moreover, it provides a necessary and sufficient condition for determining whether the determinant of a $6$-tournament is $25$.

\begin{proposition}\label{sixdd}
		Let $ T $ be a $ 6 $-tournament. Then
$$\det(T)= \begin{cases}1, & \text { if } \delta_T=0;\\ 1 \text{ or } 9, &\text { if } \delta_T\in \{3,4\};\\25, &\text { if }  \delta_T=5;\\ 49 \text{ or } 81, &\text { if }  \delta_T=6.\end{cases}$$
\end{proposition}

\begin{proof}
	By Proposition \ref{Done}, $\delta_T=0\Leftrightarrow T\in \mathcal{D}_1$, then $\det(T)=1$ since $T$ is even order.
	
	Now we consider $\delta_T\neq 0$, say, $3\leq \delta_T \leq 6$ by Corollary \ref{rebounddia}, then we complete the proof by the following two cases.
	
	\textbf{Case 1}: $3\leq \delta_T \leq5$.
	
	By Proposition \ref{sixtran}, there exists a switch of $T$, $T'$, contains a transitive $5$-subtournament. Let $V(T')=\{v_1,v_2,v_3,v_4,v_5,v_6\}$ such that $T'[X]$ is transitive with $v_1\rightarrow v_2\rightarrow \cdots \rightarrow v_5$, $\psi_{T'}(v_6,X)=(\alpha_1,\alpha_2,\ldots,\alpha_t)$, where $X=\{v_1,v_2,v_3,v_4,v_5\}$. Then $1\leq t\leq 5$, $|\alpha_1|+|\alpha_2|+\cdots+|\alpha_t|=5$, and $\alpha_i\alpha_{i+1}<0$ for $1\leq i\leq t-1$.
	
	\textbf{Subcase 1.1}: $t=1$.
	
	It is evident that $T'$ is a transitive tournament and $\delta_{T'}=0=\delta_T$, a contradiction.
	
	\textbf{Subcase 1.2}: $t=2$.
	
	$T'$ is switching equivalent to a transitive tournament with respect to the subset $W=X(2,\alpha_2)$, thus $\delta_{T'}=0=\delta_T$, a contradiction.	
	
	\textbf{Subcase 1.3}: $t=3$.
	
	If $D$ is a diamond in $T'$, by Lemma \ref{diainL}, there exists an ordering $\{v_1',v_2',v_3',v_4'\}$ of $V(D)$ such that $v_1'\in X(1,\alpha_1)$, $v_2'\in X(2,\alpha_2)$, $v_3'\in X(3,\alpha_3)$, $v_4'=v_6$. Moreover, if $T_s$ is a $4$-subtournament of $T'$ with vertex set $\{v_1',v_2',v_3',v_4'\}$ such that $v_1'\in X(1,\alpha_1)$, $v_2'\in X(2,\alpha_2)$, $v_3'\in X(3,\alpha_3)$, $v_4'=v_6$, then it is easy to see that $T_s$ is a diamond. Hence $\delta_T=\delta_{T'}=|\alpha_1||\alpha_2||\alpha_3|\in\{3,4\}$ by $|\alpha_1|+|\alpha_2|+|\alpha_3|=5$.
	
	Let $W$ be a subset of $V(T')$ such that $W=\emptyset$ if $\alpha_1>0$ and $W=\{v_6\}$ if $\alpha_1<0$. Then $T'$ is switching equivalent to a transitive blowup of $L_4$ with respect to $W$. Thus $T'\in \mathcal{D}_3$ by Theorem \ref{Dthree}, and $\det(T)=\det(T')\in\{1,9\}$.
	
	Therefore, $\delta_T=\delta_{T'}\in \{3,4\}$, $\det(T)=\det(T')\in \{1,9\}$ in this subcase. 
	
	\textbf{Subcase 1.4}: $t=4$.
	
	Let $W$ be a subset of $V(T')$ such that $W=X(4,\alpha_4)$ if $\alpha_1>0$ and $W=X(4,\alpha_4)\cup \{v_6\}$ if $\alpha_1<0$. Then $T'$ is switching equivalent to $L_4(|\alpha_1|+|\alpha_4|,|\alpha_2|,|\alpha_3|,1)$ with respect to $W$. By Subcase 1.3, $\delta_T=\delta_{T'}=(|\alpha_1|+|\alpha_4|)|\alpha_2||\alpha_3|\in \{3,4\}$ and $\det(T)=\det(T')=\det(L_4(|\alpha_1|+|\alpha_4|,|\alpha_2|,|\alpha_3|,1))\in \{1,9\}$ by Theorem \ref{Dthree}.
	
	\textbf{Subcase 1.5}: $t=5$.
	
	It is not difficult to check that the number of diamonds of $L_6$ is $5$. Moreover, by direct calculation, $\det(L_6)=25$. Let $W$ be a subset of $V(T')$ such that $W=\emptyset$ if $\alpha_1>0$ and $W=\{v_6\}$ if $\alpha_1<0$. Then $T'$ is switching equivalent to $L_6$ with respect to $W$. Hence, $\det(T)=\det(T')=\det(L_6)=25$ and $\delta_T=\delta_{T'}=5$.
	
	Combining the above arguments, we have $\det(T)\in \{1,9\}$ if $\delta_T\in \{3,4\}$, and $\det(T)=25$ if $\delta_T=5$.
	
	\textbf{Case 2}: $\delta_T=6$.
	
	By Proposition \ref{sixtran}, there exists a switch of $T$, $T'$, contains a transitive $4$-subtournament. Let $V(T')=\{v_1,v_2,v_3,v_4,v_5,v_6\}$ such that $T'[X]$ is transitive with $v_1\rightarrow v_2\rightarrow v_3\rightarrow v_4$, where $X=\{v_1,v_2,v_3,v_4\}$. By the proof of Lemma \ref{bounddia}, $\delta_{T'}=6$ if and only if each $5$-subtournament in $T'$ contains two diamonds.
	
	Let $\psi_{T'}(v_5,X)=(\alpha_1,\ldots,\alpha_{t_1})$, $\psi_{T'}(v_6,X)=(\beta_1,\ldots,\beta_{t_2})$. Then $1\leq t_1,t_2\leq 4$. If $t_1\leq 2$ or $t_2\leq 2$, then $T'[v_1,v_2,v_3,v_4,v_5]$ or $T'[v_1,v_2,v_3,v_4,v_6]$ contains no diamonds, a contradiction. Therefore, $t_1\geq 3$ and $t_2\geq 3$. Take
	 $$A_1=\{(2,-1,1),(-2,1,-1)\},A_2=\{(1,-1,2),(-1,1,-2)\},$$ $$A_3=\{(1,-2,1),(-1,2,-1)\},A_4=\{(1,-1,1,-1),(-1,1,-1,1)\}.$$
Then there exist $i,j\in\{1,2,3,4\}$ such that $\psi_{T'}(v_5,X)\in A_i$ and $\psi_{T'}(v_6,X)\in A_j$. 
	
	If $\psi_{T'}(v_5,X)\in A_2$, then $T'$ is switching equivalent to $T'_1$ with respect to $W=\{v_3,v_4\}$ of $V(T')$, and $\psi_{T'_1}(v_5,X)\in A_1$, where $v_3\rightarrow v_4\rightarrow v_1\rightarrow v_2$ in $T'_1$.
		
	If $\psi_{T'}(v_5,X)\in A_3$, then $T'$ is switching equivalent to $T'_2$ with respect to $W=\{v_1\}$ of $V(T')$, and $\psi_{T'_2}(v_5,X)\in A_1$, where $v_2\rightarrow v_3\rightarrow v_4\rightarrow v_1$ in $T'_2$.
	
	If $\psi_{T'}(v_5,X)\in A_4$, then $T'$ is switching equivalent to $T'_3$ with respect to $W=\{v_4\}$ of $V(T')$, and $\psi_{T'_3}(v_5,X)\in A_1$, where $v_4\rightarrow v_1\rightarrow v_2\rightarrow v_3$ in $T'_3$.
	
	Therefore, if $\psi_{T'}(v_5,X)\notin A_1$, there exists a switch of $T'$, $T''$, such that $\psi_{T''}(v_5,X)\in A_1$. Since $T'$ and its switch have the same determinant and the same number of diamonds, without loss of generality, we can assume that $\psi_{T'}(v_5,X)\in A_1$. Now we consider the following four subcases.
	
	\textbf{Subcase 2.1}: $\psi_{T'}(v_6,X)\in A_1$.
	
	In this subcase, $T'[V(T')\backslash\{v_4\}]$ contains no diamonds, a contradiction.
	
	\textbf{Subcase 2.2}: $\psi_{T'}(v_6,X)\in A_2$.
	
	It is easy to see that there is a switch of $T'$, $T''$, such that $\psi_{T''}(v_5,X)=(2,-1,1)$ and $\psi_{T''}(v_6,X)=(1,-1,2)$. The skew-adjacency matrix of $ T'' $ is	
	$$ S_{T''}=\left(\begin{array}{rrrrrr}
	0 & 1 & 1 & 1 & -1 & -1 \\
	-1 & 0 & 1 & 1 & -1 & 1 \\
	-1 & -1 & 0 & 1 & 1 & -1 \\
	-1 & -1 & -1 & 0 & -1 & -1 \\
	1 & 1 & -1 & 1 & 0 & \alpha \\
	1 & -1 & 1 & 1 & -\alpha & 0
	\end{array}\right), $$where $\alpha=1$ if $v_5\rightarrow v_6$, and $\alpha=-1$ otherwise.
	
	By direct calculation, we have $\det(T'')=81$ if $\alpha=1$, $\det(T'')=49$ otherwise, and $\det(T)\in\{49,81\}$ by $\det(T)=\det(T'')$.
	
	\textbf{Subcase 2.3}: $\psi_{T'}(v_6,X)\in A_3$.
	
	It is easy to see that there is a switch of $T'$, $T''$, such that $\psi_{T''}(v_5,X)=(2,-1,1)$ and $\psi_{T''}(v_6,X)=(1,-2,1)$. If $v_5\leftarrow v_6$ in $T''$, then $T''[V(T'')\backslash\{v_4\}]$ contains no diamonds, a contradiction. Therefore, $v_5\rightarrow v_6$ in $T''$, and the skew-adjacency matrix of $ T'' $ is
	$$ S_{T''}=\left(\begin{array}{rrrrrr}
	0 & 1 & 1 & 1 & -1 & -1 \\
	-1 & 0 & 1 & 1 & -1 & 1 \\
	-1 & -1 & 0 & 1 & 1 & 1 \\
	-1 & -1 & -1 & 0 & -1 & -1 \\
	1 & 1 & -1 & 1 & 0 & 1 \\
	1 & -1 & -1 & 1 & -1 & 0
	\end{array}\right). $$
	
	By direct calculation, we have $\det(T)=\det(T'')=49$.
	
	\textbf{Subcase 2.4}: $\psi_{T'}(v_6,X)\in A_4$.
	
	It is easy to see that there is a switch of $T'$, $T''$, such that $\psi_{T''}(v_5,X)=(2,-1,1)$ and $\psi_{T''}(v_6,X)=(1,-1,1,-1)$. If $v_5\leftarrow v_6$ in $T''$, then $T''[V(T'')\backslash\{v_3\}]$ contains no diamonds, a contradiction. Therefore, $v_5\rightarrow v_6$ in $T''$, and the skew-adjacency matrix of $ T'' $ is	
	$$ S_{T''}=\left(\begin{array}{rrrrrr}
	0 & 1 & 1 & 1 & -1 & -1 \\
	-1 & 0 & 1 & 1 & -1 & 1 \\
	-1 & -1 & 0 & 1 & 1 & -1 \\
	-1 & -1 & -1 & 0 & -1 & 1 \\
	1 & 1 & -1 & 1 & 0 & 1 \\
	1 & -1 & 1 & -1 & -1 & 0
	\end{array}\right). $$
	
	By direct calculation, we have $\det(T)=\det(T'')=49$.	
	
	Combining the above arguments, we have $\det(T)\in \{49,81\}$ if $\delta_T=6$.	
\end{proof}

Now we obtain some necessary and sufficient conditions for determining whether the determinant of a $6$-tournament is $25$.

\begin{theorem}\label{resixdd}
	Let $T$ be a $6$-tournament. Then the following assertions are equivalent:
	\item{\rm(i)} $\det(T)=25$.
	\item{\rm(ii)} $\delta_T=5$.
	\item{\rm(iii)} $T$ is switching equivalent to $L_6$.
	\item{\rm(iv)} $T\in \mathcal{D}_5\backslash \mathcal{D}_3$.
\end{theorem}

\begin{proof}
	Clearly, $\text{(i)}\Leftrightarrow \text{(ii)}$ follows from Proposition \ref{sixdd}.
	
	If $\delta_T=5$, then $T$ is switching equivalent to $L_6$ by the proof of Proposition \ref{sixdd}. Moreover, if $T$ is switching equivalent to $L_6$, then $\delta_T=5$ since the number of diamonds of $L_6$ is $5$. Thus $\text{(ii)}\Leftrightarrow\text{(iii)}$.
	
	Since the determinant of a $2$-tournament is $1$, the determinant of a $4$-tournament is $9$ if it is a diamond and $1$ otherwise\cite{DIACHA}, and the determinant of a tournament with odd order is $0$, we have $\det(T)=25$ if and only if $T\in \mathcal{D}_5\backslash \mathcal{D}_3$. Thus $\text{(i)}\Leftrightarrow\text{(iv)}$.
	
	Combining the above arguments, we complete the proof.
\end{proof}

\section{Characterization of $\mathcal{D}_5$}\label{sec-unicyclic}
\hspace{1.5em}In this section, we characterize $\mathcal{D}_5$. The following proposition show that a transitive blowup of a tournament $R$ in $\mathcal{D}_k$ if $R$ in $\mathcal{D}_k$. 

\begin{proposition}{\rm(\!\!\text{ Boussaïri et al.}\cite{DTHREE})}\label{blowup}
	Let $R$ be an $n$-tournament with vertex set $V=\left\{v_1, \ldots, v_n\right\}$ and $a_1, \ldots, a_n$ be a sequence of positive integers, $R(a_1,\ldots,a_n)$ be a transitive $(a_1,\ldots,a_n)$-blowup of $R$. Then
	$$
	\det(R(a_1, \ldots, a_n))=\det(R[U]),
	$$
	where $U=\{v_i\in V: a_i \text{ is odd }\}$. In particular, if $R \in \mathcal{D}_k$, then $R\left(a_1, \ldots, a_n\right)\in \mathcal{D}_k$.
\end{proposition}

\begin{remark}\label{reblowup}
	In fact, Proposition \ref{blowup} needs $U\neq \emptyset$. If $U= \emptyset$, say, all $a_i\text{ }(i=1,\ldots,n)$ are even, by using the same method in {\rm\cite{DTHREE}}, we have $\det(R(a_1, \ldots, a_n))=1$.
\end{remark}

\begin{theorem}\label{blowupclass}
	Let $R$ be an $n$-tournament, $R(a_1,\ldots,a_n)$ be a transitive blowup of $R$. Then $R\in \mathcal{D}_k$ if and only if $R(a_1,\ldots,a_n)\in \mathcal{D}_k$.  
\end{theorem}

\begin{proof}
	If $R\in \mathcal{D}_k$, then $R(a_1,\ldots,a_n)\in \mathcal{D}_k$ by Proposition \ref{blowup}.
	
	If $R(a_1,\ldots,a_n)\in \mathcal{D}_k$, we note that $R(a_1,\ldots,a_n)$ contains $R(1,\ldots,1)=R$ as a subtournament, then $R\in \mathcal{D}_k$.
\end{proof}

\begin{remark}\label{reblowupclass}
	By Theorem \ref{blowupclass}, if there exists a sequence of positive integers $a_1,\ldots,a_n$ such that $R(a_1,\ldots,a_n)\in \mathcal{D}_k$, then $R(b_1,\ldots,b_n)\in \mathcal{D}_k$ for any sequence of positive integers $b_1,\ldots,b_n$.
\end{remark}

\begin{proposition}\label{ninedet}
  Let $ R $ be an $n$-tournament ($n\geq 2$) with vertices $v_1,\cdots,v_n$, $ T_1,\cdots,T_n $ be tournaments. If there exists $ T_i $ such that $ T_i $ is not transitive for some $i$ $(1\leq i\leq n)$, then there exists a subtournament $ R_s $ of $ R(T_1,\cdots,T_n) $ such that $\det(R_s)=9\cdot\det(R)$. 
\end{proposition}

\begin{proof}
	Without loss of generality, we assume that $ T_1 $ is not transitive, then there exist $ w_1, w_2, w_3\in V[T_1]$ such that $T_1[w_1,w_2,w_3]$ is a $3$-cycle. Let $W=\{v_i\in V(R)\hspace{0.15cm}|\hspace{0.15cm}v_1\rightarrow v_i\}$. Then $R$ is switching equivalent to $R'$ with respect to $W$ such that $v_j\rightarrow v_1$ for all $2\leq j\leq n$ in $R'$. Thus $R(T_1,\ldots,T_n)$ and $R'(T_1,\ldots,T_n)$ are switching equivalent, and $V(T_j)\rightarrow V(T_1)$ for all $2\leq j\leq n$ in $R'(T_1,\ldots,T_n)$.
	
	Let $u_j\in V(T_j)$ for $2\leq j\leq n$, $R_s=R(T_1,\ldots,T_n)[w_1,w_2,w_3,u_2,\ldots,u_n]$ and $R_s'=R'(T_1,\ldots,T_n)[w_1,w_2,w_3,u_2,\ldots,u_n]=R'(T_1,\ldots,T_n)[u_2,\ldots,u_n]\rightarrow R'(T_1,\ldots,T_n)[w_1,w_2,\\w_3]$. Then $R_s$ and $R'_s$ are switching equivalent, and thus $\det(R_s)=\det(R'_s)$ by Proposition \ref{Minors}.
	
	When $n$ is odd, we have $\det(R)=0$ and $\det(R_s)=\det(R_s')=0$ since $n+2$ is odd, then $\det(R_s)=9\cdot\det(R)$.
	
	When $n$ is even, then $n-1$ and $3$ are odd, thus by (ii) of Theorem \ref{Djoin} we have $\det(R_s)=\det(R_s')=\det(R'(T_1,\ldots,T_n)[w_1,w_2,w_3]^{+})\cdot\det(R'(T_1,\ldots,T_n)[u_2,\ldots,u_n]^{+})$. Noting that $R'(T_1,\ldots,T_n)[w_1,w_2,w_3]^{+}$ is a diamond, we have $\det(R'(T_1,\ldots,T_n)[w_1,w_2,w_3]^{+})=9$, and thus $\det(R_s)=9\cdot\det(R'(T_1,\ldots,T_n)[u_2,\ldots,u_n]^{+})=9\cdot \det(R')=9\cdot\det(R)$.
	
	We complete the proof.	
\end{proof}

\begin{remark}\label{reninedet}
	Let $k\geq 3$ be an odd integer, $R$ be an $n$-tournament with vertices $v_1,v_2,\ldots,v_n$ such that $R\in \mathcal{D}_k\backslash \mathcal{D}_{k-2}$ and $\det(R)=k^2$. By using Proposition \ref{blowup} and Proposition \ref{ninedet}, it is not difficult to see that $R(T_1,\ldots,T_n)\in \mathcal{D}_k\backslash \mathcal{D}_{k-2}$ if and only if $T_i$ is transitive for every $i\in\{1,\ldots,n\}$, or equivalently, $R(T_1,\ldots,T_n)$ is a transitive blowup of $R$.
\end{remark}

 Let $T$ be an $n$-tournament, $u_1,u_2\in V(T)$ such that $u_1\neq u_2$. We say $u_1$ and $u_2$ are \textit{covertices} in $T$ if for any $v\in V(T)\backslash\{u_1,u_2\}$, $u_1\rightarrow v$ if and only if $u_2\rightarrow v$; we say $u_1$ and $u_2$ are \textit{revertices} in $T$ if for any $v\in V(T)\backslash\{u_1,u_2\}$, $u_1\rightarrow v$ if and only if $v\rightarrow u_2$. The following proposition shows an important property on $L_6$, which is crucial for the characterization of $\mathcal{D}_5\backslash \mathcal{D}_3$.

\begin{proposition}\label{crforL6}
  Let $T$ be a $7$-tournament with vertices $v_1,v_2,v_3,v_4,v_5,v_6,v_7$ such that $T[v_1,v_2,v_3,v_4,v_5,v_6]$ is $L_6$. Then $T\in \mathcal{D}_5\backslash \mathcal{D}_3$ if and only if there exists $i$ $(\leq6)$ such that $v_7$ and $v_i$ are covertices or revertices in $T$.
\end{proposition}

\begin{proof}
	If there exists $i$ $(\leq 6)$ such that $v_7$ and $v_i$ are covertices in $T$, then $T$ is a transitive blowup of $L_6$; if there exists $i(\leq 6)$ such that $v_7$ and $v_i$ are revertices in $T$, then $T$ is switching equivalent to a transitive blowup of $L_6$ with  respect to $W=\{v_7\}$. Since $L_6\in \mathcal{D}_5\backslash\mathcal{D}_3$ by Theorem \ref{resixdd}, we have $T\in \mathcal{D}_5\backslash \mathcal{D}_3$ by Proposition \ref{blowup}.
	
	Conversely, if $T\in \mathcal{D}_5\backslash \mathcal{D}_3$, we will show that there exists $i\leq 6$ such that $v_7$ and $v_i$ are covertices or revertices in $T$. 
	
	Since $T[v_1,v_2,v_3,v_4,v_5,v_6]$ is $L_6$, without loss of generality, we can assume that $T[v_1,v_2,\\v_3,v_4,v_5]$ is transitive with $v_1\rightarrow v_2\rightarrow v_3\rightarrow v_4\rightarrow v_5$ in $T$, and $\psi_T(v_6,X)=(1,-1,1,-1,1)$, where $X=\{v_1,v_2,v_3,v_4,v_5\}$.
	
	Let $\psi_T(v_7,X)=(\alpha_1,\ldots,\alpha_t)$. Then $1\leq t\leq 5$. When $\alpha_1<0$, $T$ is switching equivalent to $T'$ with respect to $\{v_7\}$ such that $\psi_{T'}(v_7,X)=(-\alpha_1,\ldots,-\alpha_t)$. Especially, if $v_i$ and $v_7$ are covertices (or revertices) in $T'$, then $v_i$ and $v_7$ are revertices (or covertices) in $T$. Therefore, we only need to consider the following five cases with $\alpha_1>0$.
	
	\textbf{Case 1}: $t=1,\alpha_1>0$.
	
	It is clear that $\alpha_1=5$. In this case, $v_1$ and $v_7$ are covertices in $T$ if $v_6\rightarrow v_7$, and $v_5$ and $v_7$ are revertices in $T$ if $v_6\leftarrow v_7$.
	
	\textbf{Case 2}: $t=2,\alpha_1>0$.
	
	Let $i=\alpha_1$. It is clear that $v_i\in X(1,\alpha_1)$, $v_{i+1}\in X(2,\alpha_2)$, and there exists $j\in \{i,i+1\}$ such that $v_6\rightarrow \{v_j,v_7\}$ or $v_6\leftarrow \{v_j,v_7\}$. Take $v_k\in \{v_i,v_{i+1}\}\backslash\{v_j\}$, it is obvious that $v_k$ and $v_7$ are revertices in $T$ in this case.	
	
	\textbf{Case 3}: $t=3,\alpha_1>0$.
	
	Clearly, $|\alpha_1\alpha_2\alpha_3|\in\{3,4\}$ by $|\alpha_1|+|\alpha_2|+|\alpha_3|=5$. We consider the following two subcases.
	
	\textbf{Subcase 3.1}: $|\alpha_1\alpha_2\alpha_3|=3$.
	
	In this subcase, there exists $i$ such that $|\alpha_i|=3$ and $|\alpha_j|=1$ for all $j\in\{1,2,3\}\backslash\{i\}$.
	
	\textbf{Subcase 3.1.1}: $|\alpha_1|=3,|\alpha_2|=|\alpha_3|=1$.
	
	If $v_6\rightarrow v_7$, we denote $T_s=T[v_2,v_3,v_4,v_5,v_6,v_7]$, $X_1=\{v_2,v_3,v_4,v_5\}$. Then $T_s$ is switching equivalent to $T_s'$ with respect to $W=\{v_6\}$ such that $T_s'[X_1]$ is transitive, $\psi_{T_s'}(v_6,X_1)=(1,-1,1,-1)$, $\psi_{T_s'}(v_7,X_1)=(2,-1,1)$ and $v_7\rightarrow v_6$ in $T_s'$. By Subcase 2.4 of the proof of Proposition \ref{sixdd}, we have $\det(T_s)=\det(T_s')=49$, then $T\notin \mathcal{D}_5$, a contradiction.
	
	If $v_6\leftarrow v_7$, we denote $T_s=T[v_1,v_2,v_4,v_5,v_6,v_7]$, $X_2=\{v_1,v_2,v_4,v_5\}$. Then $T_s[X_2]$ is transitive, $\psi_{T_s}(v_6,X_2)=(1,-2,1)$ and $\psi_{T_s}(v_7,X_2)=(2,-1,1)$. By Subcase 2.3 of the proof of Proposition \ref{sixdd}, we have $\det(T_s)=49$, then $T\notin \mathcal{D}_5$, a contradiction.
	
	\textbf{Subcase 3.1.2}: $|\alpha_2|=3,|\alpha_1|=|\alpha_3|=1$.
	
	$T$ is switching equivalent to $T'$ with respect to $W=\{v_1,v_6,v_7\}$ such that $T'[V(T)\backslash \{v_7\}]$ is $L_6$, $T'[X]$ is transitive with $v_2\rightarrow v_3\rightarrow v_4\rightarrow v_5\rightarrow v_1$ in $T'$, $\psi_{T'}(v_6,X)=(1,-1,1,-1,1)$, $\psi_{T'}(v_7,X)=(3,-1,1)$. Then $T'$ satisfies the conditions of Subcase 3.1.1, by the similar proof of Subcase 3.1.1, we have $T\notin \mathcal{D}_5$, a contradiction.
	
	\textbf{Subcase 3.1.3}: $|\alpha_3|=3,|\alpha_1|=|\alpha_2|=1$.
	
	$T$ is switching equivalent to $T''$ with respect to $W=\{v_3,v_4,v_5,v_6,v_7\}$ such that $T''[V(T)\backslash \{v_7\}]$ is $L_6$, $T''[X]$ is transitive with $v_3\rightarrow v_4\rightarrow v_5\rightarrow v_1\rightarrow v_2$ in $T''$, $\psi_{T''}(v_6,X)=(1,-1,1,-1,1)$ and $\psi_{T''}(v_7,X)=(3,-1,1)$. Then similar to Subcase 3.1.2, we have $T\notin \mathcal{D}_5$, a contradiction.
	
	\textbf{Subcase 3.2}: $|\alpha_1\alpha_2\alpha_3|=4$.
	
	In this subcase, there exists $i$ such that $|\alpha_i|=1$ and $|\alpha_j|=2$ for all $j\in\{1,2,3\}\backslash\{i\}$.
	
	\textbf{Subcase 3.2.1}: $|\alpha_3|=1,|\alpha_1|=|\alpha_2|=2$.
	
	Let $T_s=T[v_2,v_3,v_4,v_5,v_6,v_7]$, $X_1=\{v_2,v_3,v_4,v_5\}$. $T_s$ is switching equivalent to $T_s'$ with respect to $W=\{v_5,v_6,v_7\}$ such that $T_s'[X_1]$ is transitive with $v_5\rightarrow v_2\rightarrow v_3\rightarrow v_4$ in $T_s'$, $\psi_{T_s'}(v_6,X_1)=(2,-1,1)$ and $\psi_{T_s'}(v_7,X_1)=(1,-1,2)$. By Subcase 2.2 of the proof of Proposition \ref{sixdd}, we have $\det(T_s')\in\{49,81\}$, then $T\notin \mathcal{D}_5$, a contradiction.
	
	\textbf{Subcase 3.2.2}: $|\alpha_1|=1,|\alpha_2|=|\alpha_3|=2$ or $|\alpha_2|=1,|\alpha_1|=|\alpha_3|=2$.
	
	Let $W=\{v_1,v_6,v_7\}$ if $|\alpha_1|=1$, $W=\{v_4,v_5,v_7\}$ if $|\alpha_2|=1$. $T$ is switching equivalent to $T'$ with respect to $W$ such that $T'[X]$ is transitive($v_2\rightarrow v_3\rightarrow v_4\rightarrow v_5\rightarrow v_1$ in $T'$ if $|\alpha_1|=1$, $v_4\rightarrow v_5\rightarrow v_1\rightarrow v_2\rightarrow v_3$ in $T'$ if $|\alpha_2|=1$), $\psi_{T'}(v_6,X)=(1,-1,1,-1,1)$, $\psi_{T'}(v_7,X)=(2,-2,1)$. By the similar proof of Subcase 3.2.1, there is a subtournament $T'_s$ of $T'$ such that $\det(T'_s)\in \{49,81\}$, then $T'\notin \mathcal{D}_5$ and thus $T\notin \mathcal{D}_5$, a contradiction. 
	
	By Subcase 3.1 and Subcase 3.2, $t=3$ is contradictory to $T\in \mathcal{D}_5\backslash \mathcal{D}_3$.
	
	\textbf{Case 4}: $t=4,\alpha_1>0$.
	
	In this case, there exists $i$ such that $|\alpha_i|=2$ and $|\alpha_j|=1$ for all $j\in\{1,2,3,4\}\backslash\{i\}$. Let $X(i,\alpha_i)(i=1,2,3,4)$ be the corresponding vertex sets in $\psi_T(v_7,X)=(\alpha_1,\alpha_2,\alpha_3,\alpha_4)$, $W=\{v_1,v_2,v_7\}$ if $\alpha_1=2$ and $W=\{v_1,v_6,v_7\}$ if $\alpha_1=1$. Then $T$ is switching equivalent to $T'$ with respect to $W$ such that $T'[X]$ is transitive with $X(2,\alpha_2)\rightarrow X(3,\alpha_3)\rightarrow X(4,\alpha_4)\rightarrow X(1,\alpha_1)$ in $T'$. Moreover, $\psi_{T'}(v_6,X)=(1,-1,1,-1,1)$ and $\psi_{T'}(v_7,X)\in\{(1,-1,3),(1,-2,2),(2,-1,2)\}$. Then by the similar proof of Case 3, there is a subtournament $T'_s$ of $T'$ such that $\det(T'_s)\in \{49,81\}$, we have $T'\notin \mathcal{D}_5$, then $T\notin \mathcal{D}_5$, a contradiction.
	
	\textbf{Case 5}: $t=5,\alpha_1>0$.
	
	In this case, it is easy to see that $v_6$ and $v_7$ are covertices in $T$.
	
	Combining the above arguments, there exists $i\leq6$ such that $v_7$ and $v_i$ are covertices or revertices in $T$ if $T\in \mathcal{D}_5\backslash \mathcal{D}_3$. Hence we complete the proof.
\end{proof}

By the proof of Proposition \ref{crforL6}, we obtain the following corollaries, which follows notation in Proposition \ref{crforL6}.
\begin{corollary}\label{recrforL6}
	Let $T$ be a $7$-tournament with vertices $v_1,v_2,v_3,v_4,v_5,v_6,v_7$ such that $T[v_1,v_2,v_3,v_4,v_5,v_6]$ is $L_6$, $T[X]$ is transitive, $\psi_T(v_6,X)=(1,-1,1,-1,1)$ and $\psi_T(v_7,X)=(\alpha_1,\ldots,\alpha_t)$, where $X=\{v_1,\ldots,v_5\}$. Then we have
	
	\item{\rm(i)} if $t\in \{1,2,5\}$, then there exists $v_i\text{ }(1\leq i\leq 6)$ such that $v_7$ and $v_i$ are covertices or revertices in $T$;
	\item{\rm(ii)} if $t\in \{3,4\}$, then $T\notin \mathcal{D}_5$.
\end{corollary}

\begin{corollary}\label{cronlyone}
	Let $T$ be a $7$-tournament with vertices $v_1,v_2,v_3,v_4,v_5,v_6,v_7$ such that $T[v_1,v_2,v_3,v_4,v_5,v_6]$ is $L_6$. If $v_i(i\leq 6)$ and $v_7$ are covertices or revertices in $T$, then for any $j\in \{1,2,\ldots,6\}\backslash\{i\}$, $v_j$ and $v_7$ are not covertices or revertices in $T$. 
\end{corollary}

\begin{proof}
	We assume that there exists $j\in \{1,2,\ldots,6\}\backslash\{i\}$ such that $v_j$ and $v_7$ are covertices or revertices in $T$. 
	
	Let $W$ be a subset of $\{v_i,v_j\}$ such that $v\in W$ if and only if $v$ and $v_7$ are revertices in $T$. Then $T$ is switching equivalent to $T'$ with respect to $W$ such that $v_i$ and $v_7$ are covertices in $T'$, $v_j$ and $v_7$ are covertices in $T'$.
	
	Then $v_i$ and $v_j$ are covertices in $T'[Y]$, where $Y=\{v_1,v_2,v_3,v_4,v_5,v_6\}$, and thus there exists a $5$-tournament $R$ such that $T'[Y]$ is $R(1,1,1,1,2)$. By Proposition \ref{blowup}, $\det(T'[Y])=\det(R(1,1,1,1,2))\leq 9< 25$, which contradicts $\det(T'[Y])=\det(T[Y])=\det(L_6)=25$.
\end{proof}

Let $T$ be a tournament, $\{v_1,v_2\}\in V(T)$. We write $\theta_T(v_1,v_2)=1$ if $v_1\rightarrow v_2$ in $T$, and write $\theta_T(v_1,v_2)=-1$ otherwise.

\begin{proposition}\label{mustcoL6}
	Let $ T $ be an $8$-tournament with $V(T)=\{v_1,v_2,v_3,v_4,v_5,v_6,u_1,u_2\}$, $X=\{v_1,v_2,v_3,v_4,v_5\}$ such that $T[v_1,v_2,v_3,v_4,v_5,v_6]$ is $L_6$, where $T[X]$ is transitive with $v_1\rightarrow v_2\rightarrow v_3\rightarrow v_4\rightarrow v_5$ in $T$, $\psi_T(v_6,X)=(1,-1,1,-1,1)$, $u_1$ and $v_i$ be covertices in $T[V(T)\backslash\{u_2\}]$ for some $i\text{ }(1\leq i\leq 6)$, $u_2$ and $v_j$ be covertices in $T[V(T)\backslash\{u_1\}]$ for some $j\in\{1,2,3,4,5,6\}\backslash\{i\}$. Then $T\in \mathcal{D}_5\backslash \mathcal{D}_3$ if and only if $\theta_T(u_1,u_2)\cdot\theta_T(v_i,v_j)=1$.
\end{proposition}

\begin{proof}
	If $\theta_T(u_1,u_2)\cdot\theta_T(v_i,v_j)=1$, then $u_1\rightarrow u_2$ if and only if $v_i\rightarrow v_j$, thus $u_1$ and $v_i$ are covertices in $T$, $u_2$ and $v_j$ are covertices in $T$ since $u_1$ and $v_i$ are covertices in $T[V(T)\backslash\{u_2\}]$, $u_2$ and $v_j$ are covertices in $T[V(T)\backslash\{u_1\}]$. Therefore, $T$ is a transitive blowup of $L_6$, and $T\in \mathcal{D}_5\backslash \mathcal{D}_3$ by Proposition \ref{blowup} and Theorem 2.11. 
	
	Conversely, if $T\in \mathcal{D}_5\backslash \mathcal{D}_3$, we will show $\theta_T(u_1,u_2)\cdot\theta_T(v_i,v_j)=1$. Without loss of generality, we assume that $i<j$, then we have $i\in\{1,2,3,4,5\}$ and $j\in \{2,3,4,5,6\}$.
	
	\textbf{Case 1}: $i=1,j=2$.
	
	Clearly, $\theta_T(v_1,v_2)=1$ by $T[X]$ is transitive with $v_1\rightarrow v_2\rightarrow v_3\rightarrow v_4\rightarrow v_5$ in $T$. We assume that $\theta_T(u_1,u_2)=-1$, now we show that there exists a contradiction.
	
	\textbf{Subcase 1.1}: $u_1\rightarrow v_1$.
	
	Let $T_1=T[V(T)\backslash\{v_2\}]$, $X_1=\{v_1,u_2,v_3,v_4,v_5\}$. Since $u_1$ and $v_1$ are covertices in $T[V(T)\backslash\{u_2\}]$, $u_2$ and $v_2$ are covertices in $T[V(T)\backslash\{u_1\}]$, and $T[v_1,v_2,v_3,v_4,v_5,v_6]$ is $L_6$, it is clear that $T_1[V(T_1)\backslash\{u_1\}]$ is $L_6$, $T_1[X_1]$ is transitive with $v_1\rightarrow u_2\rightarrow v_3\rightarrow v_4\rightarrow v_5$ in $T_1$, $\psi_{T_1}(v_6,X_1)=(1,-1,1,-1,1)$ and $\psi_{T_1}(u_1,X_1)=(1,-1,3)$. By Proposition \ref{crforL6} and Corollary \ref{recrforL6}, we have $T_1\notin \mathcal{D}_5$, and thus $T\notin \mathcal{D}_5$, a contradiction.
	
	\textbf{Subcase 1.2}: $u_1\leftarrow v_1$.
	
	\textbf{Subcase 1.2.1}: $v_2\rightarrow u_2$.
	
	Let $T_2=T[V(T)\backslash\{v_1\}]$, $X_2=\{u_1,v_2,v_3,v_4,v_5\}$. It is clear that $T_2[V(T_2)\backslash \{u_2\}]$ is $L_6$, $T_2[X_2]$ is transitive with $u_1\rightarrow v_2\rightarrow v_3\rightarrow v_4\rightarrow v_5$ in $T_2$, $\psi_{T_2}(v_6,X_2)=(1,-1,1,-1,1)$, $\psi_{T_2}(u_2,X_2)=(1,-1,3)$. Similarly, by Proposition \ref{crforL6} and Corollary \ref{recrforL6}, $T_2\notin \mathcal{D}_5$, and thus $T\notin \mathcal{D}_5$, a contradiction. 
	
	\textbf{Subcase 1.2.2}: $v_2\leftarrow u_2$. 
	
	Let $X_3=\{v_1,v_2,v_3,v_4,v_5,u_1,u_2\}$. Then $T[X_3]$ is transitive with $v_1\rightarrow u_2\rightarrow u_1\rightarrow v_2\rightarrow v_3\rightarrow v_4\rightarrow v_5$ in $T$, $\psi_T(v_6,X_3)=(1,-1,1,-1,1,-1,1)$, thus $T$ is $L_8$. Therefore, according to the ordering $v_1,u_2,u_1,v_2,v_3,v_4,v_5,v_6$, we have	
	$$ \det(T)=\left|\begin{array}{rrrrrrrr}
	0 & 1 & 1 & 1 & 1 & 1 & 1 & -1 \\
	-1 & 0 & 1 & 1 & 1 & 1 & 1 & 1 \\
	-1 & -1 & 0 & 1 & 1 & 1 & 1 & -1\\
	-1 & -1 & -1 & 0 & 1 & 1 & 1 & 1 \\
	-1 & -1 & -1 & -1 & 0 & 1 & 1 & -1 \\
	-1 & -1 & -1 & -1 & -1 & 0 & 1 & 1 \\
	-1 & -1 & -1 & -1 & -1 & -1 & 0 & -1 \\
	1 & -1 & 1 & -1 & 1 & -1 & 1 & 0\\
	\end{array}\right|. $$		
	By direct calculation, $\det(T)=49$. Hence $T\notin\mathcal{D}_5$, a contradiction.
	
	Combining the above arguments, we have $\theta_T(u_1,u_2)\neq -1$, then $\theta_T(u_1,u_2)\cdot\theta_T(v_i,v_j)=1$ for $i=1$ and $j=2$.
	
	\textbf{Case 2}: $i=1,j\in\{3,4\}$.
	
	Clearly, we have $\theta_T(v_1,v_j)=1$ for $j\in\{3,4\}$, and we assume that $\theta_T(u_1,u_2)=-1$.
	
	Let $T_2,X_2$ be defined as in Subcase 1.2.1. Then $\psi_{T_2}(u_2,X_2)\in\{(1,-2,2),(1,-1,3),(1,-3,1)\}$, thus $T_2\notin \mathcal{D}_5$ by Proposition \ref{crforL6} and Corollary \ref{recrforL6}, a contradiction.
	
	\textbf{Case 3}: $i=1,j=5$.
	
	Clearly, $\theta_T(v_1,v_5)=1$, and we assume that $\theta_T(u_1,u_2)=-1$.
	
	Let $W=\{u_2,v_5,v_6\}$. Then $T$ is switching equivalent to $T'$ with respect to $W$ such that $T'[v_1,v_2,v_3,v_4,v_5]$ is transitive with $v_5\rightarrow v_1\rightarrow v_2\rightarrow v_3\rightarrow v_4$ in $T'$.
	
	Let $v_1'=v_5$, $v_i'=v_{i-1}(2\leq i\leq5)$, $v_6'=v_6$, $u_1'=u_2$, $u_2'=u_1$. Then $T'[v_1',v_2',v_3',v_4',v_5',v_6']$ is $L_6$, $T'[v_1',v_2',v_3',v_4',v_5']$ is transitive, $v_1'\rightarrow v_2'\rightarrow v_3'\rightarrow v_4'\rightarrow v_5'$ in $T'$ and $u_1'\leftarrow u_2'$ in $T'$. Moreover, $u_1'$ and $v_1'$ are covertices in $T'[u_1',v_1',v_2',v_3',v_4',v_5',v_6']$, $u_2'$ and $v_2'$ are covertices in $T'[u_2',v_1',v_2',v_3',v_4',v_5',v_6']$.
	
	By Case 1, we have $T'\notin \mathcal{D}_5$, then $T\notin \mathcal{D}_5$ by Proposition \ref{Minors}, a contradiction.
	
	\textbf{Case 4}: $i=1,j=6$.
	
	Clearly, $\theta_T(v_1,v_6)=-1$, and we assume that $\theta_T(u_1,u_2)=1$.
	
	Let $T_2,X_2$ defined as in Subcase 1.2.1. Then $\psi_{T_2}(u_2,X_2)=(-2,1,-1,1)$. Let $T_2'$ be a switch of $T_2$ with respect to $W=\{u_2\}$. Then $\psi_{T_2'}(u_2,X_2)=(2,-1,1,-1)$. Hence $T_2'\notin \mathcal{D}_5$ by Proposition \ref{crforL6} and Corollary \ref{recrforL6}, and $T_2\notin \mathcal{D}_5$, which implies $T\notin \mathcal{D}_5$, a contradiction.
    
    \textbf{Case 5}: $i\in\{2,3,4,5\}$.
    
    Let $W=\{v_1,v_2,\ldots,v_{i-1}\}$ if $i$ is odd, $W=\{v_1,v_2,\ldots,v_{i-1},v_6\}$ if $i$ is even and $j\neq 6$, $W=\{v_1,v_2,\ldots,v_{i-1},v_6,u_2\}$ if $i$ is even and $j=6$. Then $T$ is switching equivalent to $T'$ with respect to $W$ such that $T'[v_1,v_2,\ldots,v_6]$ is $L_6$, $T'[X]$ is transitive with $v_i\rightarrow v_{i+1}\rightarrow \cdots \rightarrow v_5\rightarrow v_1\rightarrow \cdots \rightarrow v_{i-1}$ in $T'$, $\psi_{T'}(v_6,X)=(1,-1,1,-1,1)$. Clearly, $T'\in \mathcal{D}_5\backslash\mathcal{D}_3$ by $T\in \mathcal{D}_5\backslash \mathcal{D}_3$ and Proposition \ref{Minors}. Moreover, if $\theta_{T'}(u_1,u_2)\cdot\theta_{T'}(v_i,v_j)=1$ in $T'$, then $\theta_{T}(u_1,u_2)\cdot\theta_{T}(v_i,v_j)=1$ in $T$ by the construction of $W$. So we only need to show $\theta_{T'}(u_1,u_2)\cdot\theta_{T'}(v_i,v_j)=1$ in $T'$.
    
    Let $v_1'=v_i,v_2'=v_{i+1},\cdots,v_{5-i+1}'=v_5,v_{5-i+2}'=v_1,\cdots,v_5'=v_{i-1},v_6'=v_6$. Then there exists $k>1$ such that $v_k'=v_j$. Let $Z=\{v_1',v_2',v_3',v_4',v_5',v_6'\}$. Then $T'[Z]$ is $L_6$ with $v_1'\rightarrow v_2'\rightarrow v_3'\rightarrow v_4'\rightarrow v_5'$ in $T'[Z]$, $u_1$ and $v_1'$ are covertices in $T'[Z\cup\{u_1\}]$, $u_2$ and $v_k'$ are covertices in $T'[Z\cup\{u_2\}]$. By Case 1 $\sim$ Case 4 and $T'\in \mathcal{D}_5\backslash\mathcal{D}_3$, we have $\theta_{T'}(u_1,u_2)\cdot \theta_{T'}(v'_1,v'_k)=1$, or equivalently, $\theta_{T'}(u_1,u_2)\cdot \theta_{T'}(v_i,v_j)=1$, which implies $\theta_T(u_1,u_2)\cdot \theta_T(v_i,v_j)=1$.

    Combining the above five cases, if $T\in \mathcal{D}_5\backslash \mathcal{D}_3$, then we have $\theta_T(u_1,u_2)\cdot \theta_T(v_i,v_j)=1$, and we complete the proof.
\end{proof}

Now we characterize the set $\mathcal{D}_5\backslash \mathcal{D}_3$.

\begin{theorem}\label{D5D3character}
  Let $ T $ be a tournament. Then $ T\in\mathcal{D}_5 \backslash \mathcal{D}_3 $ if and only if $ T $ is switching equivalent to a transitive blowup of $ L_6 $. 
\end{theorem}

\begin{proof}
	If $ T $ is switching equivalent to a transitive blowup of $ L_6 $, since $L_6\in \mathcal{D}_5\backslash \mathcal{D}_3$, then $T\in \mathcal{D}_5\backslash \mathcal{D}_3$ by Proposition \ref{blowup} and Theorem \ref{blowupclass}. 
	
	Conversely, if $T\in \mathcal{D}_5\backslash \mathcal{D}_3$, we will show that $T$ is switching equivalent to a transitive blowup of $L_6$. 
	
	Since $T\in \mathcal{D}_5\backslash \mathcal{D}_3$, there exists a $6$-subtournament $T_s$ such that $\det(T_s)=25$ by (iii) of Theorem \ref{Dthree}. If $T_s$ is not $L_6$, then there is a switch of $T$, $T^{*}$, where $T^{*}[V(T_s)]$ is $L_6$ by Theorem \ref{resixdd}. Therefore, without loss of generality, we can assume that $T_s$ is $L_6$. 
	
	Let $|V(T)|=n$ and $V(T)=\{v_1,v_2,\ldots,v_n\}$ such that $T_s=T[v_1,v_2,v_3,v_4,v_5,v_6]$, $T_s[X_s]$ is transitive with $v_1\rightarrow v_2\rightarrow v_3\rightarrow v_4\rightarrow v_5$ in $T_s$ and $\psi_{T_s}(v_6,X_s)=(1,-1,1,-1,1)$, where $X_s=\{v_1,v_2,v_3,v_4,v_5\}$, $T_s^{(k)}=T[V(T_s)\cup\{v_k\}]$ for $7\leq k\leq n$. Then there exists $v_i\in V(T_s)$ such that $v_i$ and $v_k$ are covertices or revertices in $T_s^{(k)}$ for $k\in \{7,\ldots,n\}$ by Proposition \ref{crforL6} .
	
	Let $1\leq i\leq 6$, $Y_i=\{v_k\hspace{0.15cm}|\hspace{0.15cm} v_k\text{ and }v_i\text{ are covertices or revertices in }T_s^{(k)},7\leq k\leq n\}$, $X_i=Y_i\cup\{v_i\}$. Then $\bigcup\limits_{i=1}\limits^6 X_i=V(T)$, $X_i\cap X_j=\emptyset$ for any $1\leq i<j \leq 6$ by Corollary \ref{cronlyone}, and thus $X_1,X_2,\ldots,X_6$ form a partition of $V(T)$.
	
	Let $W_i=\{v_k\hspace{0.15cm}|\hspace{0.15cm} v_k\text{ and }v_i\text{ are revertices in }T_s^{(k)},7\leq k\leq n\}$ for $1\leq i\leq 6$, $W=\bigcup\limits_{i=1}\limits^6 W_i$. Then $T$ is switching equivalent to $T'$ with respect to $W$ such that $T'[V(T_s)]$ is also $L_6$ (if $W=\emptyset$, then $T'=T$), but for all $v_k\in Y_i$, $v_k$ and $v_i$ are covertices in $T'[V(T_s^{(k)})]$ for $1\leq i\leq 6, 7\leq k\leq n$. Moreover, for $1\leq i<j\leq 6$, $v_{k_1}\in Y_i$ and $v_{k_2}\in Y_j$, we have $\theta_{T'}(v_{k_1},v_{k_2})\cdot \theta_{T'}(v_i,v_j)=1$ by Proposition \ref{mustcoL6} and $T'\in \mathcal{D}_5\backslash \mathcal{D}_3$, or equivalently, $v_{k_1}\rightarrow v_{k_2}$ if and only if $v_i\rightarrow v_j$, which implies $X_i\rightarrow X_j$ if $v_i\rightarrow v_j$, and $X_i\leftarrow X_j$ if $v_i\leftarrow v_j$.
	
	Take $V(L_6)=\{u_1,u_2,u_3,u_4,u_5,u_6\}$ such that $u_1\rightarrow u_2\rightarrow u_3\rightarrow u_4\rightarrow u_5$ and $\psi_{L_6}(u_6,\{u_1,u_2,u_3,u_4,u_5\})=(1,-1,1,-1,1)$. Then $T'=L_6(T'[X_1],T'[X_2],T'[X_3],T'[X_4],\\T'[X_5],T'[X_6])$, where $u_i$ is replaced by $T'[X_i]$ in $L_6$. Since $\det(L_6)=25$, $L_6\in \mathcal{D}_5\backslash \mathcal{D}_3$ and $T'\in \mathcal{D}_5\backslash \mathcal{D}_3$, we have all $T'[X_i]$ are transitive by Remark \ref{reninedet}, say, $T'$ is a transitive blowup of $L_6$, and $T$ is switching equivalent to a transitive blowup of $L_6$.
	
	We complete the proof.
\end{proof}

By the definition of $L_n$, we have $L_2$ is a transitive tournament, $L_4$ is a diamond. Noting that the unique two diamonds are switching equivalent, and any tournament $T$ is a transitive blowup of itself, by Proposition \ref{Done}, Theorem \ref{Dthree} and Theorem \ref{D5D3character}, we can characterize the sets $\mathcal{D}_1$, $\mathcal{D}_3$ and $\mathcal{D}_5$ as follows.  

\begin{theorem}\label{D5character}
	Let $ T $ be an $n$-tournament $(n\geq 2)$. Then we have 
	\item{\rm(i)} $T\in \mathcal{D}_1$ if and only if $T$ is switching equivalent to a transitive blowup of $L_2$.
	\item{\rm(ii)} $T\in \mathcal{D}_3\backslash \mathcal{D}_1$ if and only if $T$ is switching equivalent to a transitive blowup of $L_4$.
	\item{\rm(iii)} $T\in \mathcal{D}_3$ if and only if $T$ is switching equivalent to a transitive blowup of $L_k$, where $k\in\{2,4\}$.
	\item{\rm(iv)} $T\in \mathcal{D}_5\backslash \mathcal{D}_3$ if and only if $T$ is switching equivalent to a transitive blowup of $L_6$.
	\item{\rm(v)} $T\in \mathcal{D}_5$ if and only if $T$ is switching equivalent to a transitive blowup of $L_k$, where $k\in\{2,4,6\}$.
\end{theorem}

Let $k\geq 3$ be an odd integer. If $\mathcal{D}_k\backslash\mathcal{D}_{k-2}\neq \emptyset$, by the result of Theorem \ref{D5character}, a natural question is that whether an $n$-tournament $T\in \mathcal{D}_k\backslash\mathcal{D}_{k-2}$ if and only if $T$ is switching equivalent to a transitive blowup of $L_{k+1}$ for $k\geq 7$. 

In fact, the answer to the above question is no. For example, there exists a $6$-tournament $T$ such that $\det(T)=49$ by the proof of Subcases 2.2 $ \sim $ 2.4 of Proposition \ref{sixdd}, and then $T\in \mathcal{D}_7\backslash \mathcal{D}_5$ by a simple proof, but $T$ can not be switching equivalent to a transitive blowup of $L_8$ since $|V(L_8)|=8>|V(T)|=6$.

Therefore, we can propose the further question as follows.

\begin{question}\label{queanyDk}
	Let $k\text{ }(\geq 7)$ be odd. If $\mathcal{D}_k\backslash\mathcal{D}_{k-2} \neq \emptyset$, then what conditions must $T$ and $L_{k+1}$ satisfy such that $T\in \mathcal{D}_k\backslash \mathcal{D}_{k-2}$ if and only if $T$ is switching equivalent to a transitive blowup of $L_{k+1}$?
\end{question}

In fact, we will show that $\mathcal{D}_k\backslash \mathcal{D}_{k-2}\neq\emptyset$ for any odd $k\text{ }(\geq 3)$ in next section.

\section{The properties of $L_n$}\label{ln}
\hspace{1.5em}In this section, we will show two important properties of $L_n$, one is that $\det(L_n)=(n-1)^2$ for any even integer $n\text{ }(\geq 2)$, and the other is that $L_n\in \mathcal{D}_{n-1}\backslash \mathcal{D}_{n-3}$ for any even integer $n\text{ }(\geq 4)$, which implies $\mathcal{D}_k\backslash\mathcal{D}_{k-2}\neq\emptyset$ for any odd $k\text{ }(\geq 3)$.

In \cite{UNI}, the authors proposed the following question. 

\begin{question}{\rm(\!\!\text{ Belkouche et al.}\cite[Question 10]{UNI})}\label{queanyk}
	Does there exist a tournament whose determinant is $k^2$ for every odd number $k$?
\end{question}

In the following, we will show $\det(L_n)=(n-1)^2$, which implies the answer to Question \ref{queanyk} is yes.

\begin{lemma}\label{ledetLn}
	Let $m\geq 3$ be a positive integer,
	\begin{align*}
	Q_m=\left|\begin{array}{rrrrrrr}
	1 & 0 & 1 & 1 & \cdots & 1 & 1 \\
	-1 & -1 & 0 & 1 & \cdots & 1 & 1 \\
	1 & -1 & -1 & 0 & \cdots & 1 & 1 \\
	-1 & -1 & -1 & -1 & \cdots & 1 & 1 \\
	\vdots & \vdots & \vdots & \vdots & \ddots & \vdots & \vdots \\
	(-1)^{m-2} & -1 & -1 & -1 & \cdots & -1& 0 \\
	(-1)^{m-1} & -1 & -1 & -1 & \cdots & -1& -1
	\end{array}\right|.
	\end{align*}
	Then we have
	
	\item{\rm(i)} $Q_m= \begin{cases}2-Q_{m-1}, & \text { if } m\geq 5 \text{ and } m \text{ is odd};\\ -Q_{m-1}, &\text { if } m\geq 4 \text{ and } m \text{ is even}.\end{cases}$
	
	\item{\rm(ii)} if $m$ is odd, then $Q_m=m$.
\end{lemma}

\begin{proof}	
	If $m$ is odd and $m\geq 5$, by subtracting the $(m-1)$-th row from the $m$-th row in $Q_m$, we have
	\begin{align}
	Q_m&=\left|\begin{array}{rrrrrrr}
	1 & 0 & 1 & 1 & \cdots & 1 & 1 \\
	-1 & -1 & 0 & 1 & \cdots & 1 & 1 \\
	1 & -1 & -1 & 0 & \cdots & 1 & 1 \\
	-1 & -1 & -1 & -1 & \cdots & 1 & 1 \\
	\vdots & \vdots & \vdots & \vdots & \ddots & \vdots & \vdots \\
	-1 & -1 & -1 & -1 & \cdots & -1 & 0 \\
	2 & 0 & 0 & 0 & \cdots & 0 & -1
	\end{array}\right| \notag\\
	&= 2\times 
	\left|\begin{array}{rrrrr}
	0 & 1 & 1 & \cdots & 1 \\
	-1 & 0 & 1 & \cdots & 1 \\
	-1 & -1 & 0 & \cdots & 1 \\
	\vdots & \vdots & \vdots & \ddots & \vdots \\
	-1 & -1 & -1 & \cdots & 0
	\end{array}\right|-Q_{m-1}.\label{ak}
	\end{align}
	
	We note that the first term on the right of (\ref{ak}) equals $2$ times the determinant of a transitive $(m-1)$-tournament, which implies it equals to $2$ if $m$ is odd, and $0$ if $m$ is even. Thus
	\begin{align*}
	Q_m=2-Q_{m-1}.
	\end{align*}
	
	If $m$ is even and $m\geq 4$,
	\begin{align*}
	Q_m&=\left|\begin{array}{rrrrrrr}
	1 & 0 & 1 & 1 & \cdots & 1 & 1 \\
	-1 & -1 & 0 & 1 & \cdots & 1 & 1 \\
	1 & -1 & -1 & 0 & \cdots & 1 & 1 \\
	-1 & -1 & -1 & -1 & \cdots & 1 & 1 \\
	\vdots & \vdots & \vdots & \vdots & \ddots & \vdots & \vdots \\
	1 & -1 & -1 & -1 & \cdots & -1 & 0 \\
	-2 & 0 & 0 & 0 & \cdots & 0 & -1
	\end{array}\right| \\
	&=(-2)\times 
	\left|\begin{array}{rrrrr}
	0 & 1 & 1 & \cdots & 1 \\
	-1 & 0 & 1 & \cdots & 1 \\
	-1 & -1 & 0 & \cdots & 1 \\
	\vdots & \vdots & \vdots & \ddots & \vdots \\
	-1 & -1 & -1 & \cdots & 0
	\end{array}\right|-Q_{m-1}\\
	&=-Q_{m-1}.
	\end{align*}
	
	 So (i) holds.
	 
	 If $m=3$, we have 
	 $
	 Q_3=\left|\begin{array}{rrr}
	 1 & 0 & 1 \\
	 -1 & -1 & 0 \\
	 1 & -1 & -1 \\
	 \end{array}\right|=3
	 $. 
	 
	 For odd $m\text{ }(\geq 5)$, by (i) and induction, we have
	\begin{align*}
	Q_m=2-Q_{m-1}=2+Q_{m-2}=\cdots=m-3+Q_3=m-3+3= m.
	\end{align*}
	
	So (ii) holds.
\end{proof}

By Lemma \ref{ledetLn} and the properties of determinant, we prove that $\det(L_n)=(n-1)^2$ for any positive even integer $n$.

\begin{theorem}\label{detLn}
	Let $n$ be even. Then we have
	\item{\rm(i)} $\det(L_n)=\det(L_{n-2})+4(n-2)$ for $n\geq 4$.
	\item{\rm(ii)} $\det(L_n)=(n-1)^2$.
\end{theorem}

\begin{proof}
Take $V(L_n)=\{u_1,u_2,\ldots,u_n\}$ such that $L_n[u_1,\ldots,u_{n-1}]$ is transitive with $u_1\rightarrow u_2\rightarrow \cdots \rightarrow u_{n-1}$ in $L_n$ and $\psi_{L_n}(u_n,X)=((-1)^0,(-1)^1,\ldots,(-1)^{n-2})$, where $X=\{u_1,\ldots,u_{n-1}\}$. It is clear that $\det(L_2)=1$ and $\det(L_4)=9$, thus we only need to prove that (i) and (ii) hold for $n\geq 6$.

Let $C_n$ be an $n\times n$ matrix defined by
\begin{align*}
C_n&=\left(\begin{array}{rrrrrr}
0 & 1 & 1 & 1 & \cdots & 1 \\
-1 & 0 & 1 & 1 & \cdots & 1 \\
-1 & -1 & 0 & 1 & \cdots & 1 \\
-1 & -1 & -1 & 0 & \cdots & 1 \\
\vdots & \vdots & \vdots & \vdots & \ddots & 1 \\
-1 & -1 & -1 & -1 & \cdots & 0 \\
\end{array}\right),
\end{align*} 
$l_n$ be a $1\times n$ matrix defined by
\begin{align*}
l_n&=\left(\begin{array}{rrrrrr}
(-1)^0 & (-1)^1 & (-1)^2 & \cdots & (-1)^{n-2} & (-1)^{n-1}\\
\end{array}\right),
\end{align*} 
and $L_n'$ be a tournament such that $V(L_n')=V(L_n)$, $L_n'(X')$ is transitive with $u_2\rightarrow u_3\rightarrow \cdots \rightarrow u_n$ in $L_n'$, $\psi_{L_n'}(u_1,X')=((-1)^0,(-1)^1,\ldots,(-1)^{n-2})$, where $X'=\{u_2,\ldots,u_n\}$. Then $L_n\cong L_n'$, and
\begin{align}
\det(L_n)=\det(L_n')=\left|\begin{array}{cc}
0 & l_{n-1} \\
-l_{n-1}^t & C_{n-1}\\
\end{array}\right|. \label{ak2}
\end{align} 

Now we compute the value of $\det(L_n)$. By (\ref{ak2}), we have
\begin{align}
\det(L_n)&=\tiny{\left|\begin{array}{rrrrrrrrr}
	0 & 1 & -1 & 1 & -1 & \cdots & 1 & -1 & 1 \\
	-1 & 0 & 1 & 1 & 1 & \cdots & 1 & 1 & 1 \\
	1 & -1 & 0 & 1 & 1 & \cdots & 1 & 1 & 1 \\
	-1 & -1 & -1 & 0 & 1 & \cdots &1 & 1 & 1 \\
	1 & -1 & -1 & -1 & 0 & \cdots & 1 & 1 & 1 \\
	\vdots & \vdots & \vdots & \vdots & \vdots & \ddots & \vdots & \vdots & \vdots \\
	-1 & -1 & -1 & -1 & -1 & \cdots  & 0 & 1 & 1 \\
	1 & -1 & -1 & -1 & -1 & \cdots  & -1 & 0 & 1 \\
	-1 & -1 & -1 & -1 & -1 & \cdots & -1 & -1 & 0 \\
	\end{array}\right|} \notag\\
&=\tiny{\left|\begin{array}{rrrrrrrrr}
	0+0 & 1+0 & -1+0 & 1+0 & -1+0 & \cdots & 1+0 & 1-2 & 1+0 \\
	-1+0 & 0 & 1 & 1 & 1 & \cdots  & 1 & 1 & 1 \\
	1+0 & -1 & 0 & 1 & 1 & \cdots & 1 & 1 & 1 \\
	-1+0 & -1 & -1 & 0 & 1 & \cdots & 1 & 1 & 1 \\
	1+0 & -1 & -1 & -1 & 0 & \cdots & 1 & 1 & 1 \\
	\vdots & \vdots & \vdots & \vdots & \vdots & \ddots & \vdots & \vdots & \vdots \\
	-1+0 & -1 & -1 & -1 & -1 & \cdots & 0 & 1 & 1 \\
	-1+2 & -1 & -1 & -1 & -1 & \cdots & 1 & 0 & 1 \\
	-1+0 & -1 & -1 & -1 & -1 & \cdots & 1 & -1 & 0 \\
	\end{array}\right|} \notag\\
&=\tiny{\left|\begin{array}{rrrrrrrrr}
	0 & 1 & -1 & 1 & -1 & \cdots & 1 & 1 & 1 \\
	-1 & 0 & 1 & 1 & 1 & \cdots & 1 & 1 & 1 \\
	1 & -1 & 0 & 1 & 1 & \cdots & 1 & 1 & 1 \\
	-1 & -1 & -1 & 0 & 1 & \cdots & 1 & 1 & 1 \\
	1 & -1 & -1 & -1 & 0 & \cdots & 1 & 1 & 1 \\
	\vdots & \vdots & \vdots & \vdots & \vdots & \ddots & \vdots & \vdots & \vdots \\
	-1 & -1 & -1 & -1 & -1 & \cdots & 0 & 1 & 1 \\
	-1 & -1 & -1 & -1 & -1 & \cdots & -1 & 0 & 1 \\
	-1 & -1 & -1 & -1 & -1 & \cdots & -1 & -1 & 0 \\
	\end{array}\right| +\left|\begin{array}{rrrrrrrrr}
	0 & 1 & -1 & 1 & -1 & \cdots & 1 & 1 & 1 \\
	0 & 0 & 1 & 1 & 1 & \cdots & 1 & 1 & 1 \\
	0 & -1 & 0 & 1 & 1 & \cdots & 1 & 1 & 1 \\
	0 & -1 & -1 & 0 & 1 & \cdots & 1 & 1 & 1 \\
	0 & -1 & -1 & -1 & 0 & \cdots & 1 & 1 & 1 \\
	\vdots & \vdots & \vdots & \vdots & \vdots & \ddots & \vdots & \vdots & \vdots \\
	0 & -1 & -1 & -1 & -1 & \cdots & 0 & 1 & 1 \\
	2 & -1 & -1 & -1 & -1 & \cdots & -1 & 0 & 1 \\
	0 & -1 & -1 & -1 & -1 & \cdots & -1 & -1 & 0 \\
	\end{array}\right|} \notag\\
&+\tiny{\left|\begin{array}{rrrrrrrrr}
	0 & 0 & 0 & 0 & 0 & \cdots & 0 & -2 & 0 \\
	-1 & 0 & 1 & 1 & 1 & \cdots & 1 & 1 & 1 \\
	1 & -1 & 0 & 1 & 1 & \cdots & 1 & 1 & 1 \\
	-1 & -1 & -1 & 0 & 1 & \cdots & 1 & 1 & 1 \\
	1 & -1 & -1 & -1 & 0 & \cdots & 1 & 1 & 1 \\
	\vdots & \vdots & \vdots & \vdots & \vdots & \ddots & \vdots & \vdots & \vdots \\
	-1 & -1 & -1 & -1 & -1 & \cdots & 0 & 1 & 1 \\
	-1 & -1 & -1 & -1 & -1 & \cdots & -1 & 0 & 1 \\
	-1 & -1 & -1 & -1 & -1 & \cdots & -1 & -1 & 0 \\
	\end{array}\right|+\left|\begin{array}{rrrrrrrrr}
	0 & 0 & 0 & 0 & 0 & \cdots & 0 & -2 & 0 \\
	0 & 0 & 1 & 1 & 1 & \cdots & 1 & 1 & 1 \\
	0 & -1 & 0 & 1 & 1 & \cdots & 1 & 1 & 1 \\
	0 & -1 & -1 & 0 & 1 & \cdots & 1 & 1 & 1 \\
	0 & -1 & -1 & -1 & 0 & \cdots & 1 & 1 & 1 \\
	\vdots & \vdots & \vdots & \vdots & \vdots & \ddots & \vdots & \vdots & \vdots \\
	0 & -1 & -1 & -1 & -1 & \cdots & 0 & 1 & 1 \\
	2 & -1 & -1 & -1 & -1 & \cdots & -1 & 0 & 1 \\
	0 & -1 & -1 & -1 & -1 & \cdots & -1 & -1 & 0 \\ 
	\end{array}\right|}\label{ak3}\\ 
&=\det(A_1)+\det(A_2)+\det(A_3)+\det(A_4), \label{ak4}
\end{align}
where $\det(A_i)$ is the $i$-th determinant in (\ref{ak3}).

Since $n$ is even, we have 
\begin{align}
	\det(A_2)&=\det(A_2^{t})=(-1)^n\det(-A_2^{t})=\det(A_3) \notag\\
			&=(-2)\times(-1)^{1+n-1}\times\tiny{\left|\begin{array}{rrrrrrrrr}
				-1 & 0 & 1 & 1 & 1 & \cdots & 1 & 1 & 1 \\
				1 & -1 & 0 & 1 & 1 & \cdots & 1 & 1 & 1 \\
				-1 & -1 & -1 & 0 & 1 & \cdots & 1 & 1 & 1 \\
				1 & -1 & -1 & -1 & 0 & \cdots & 1 & 1 & 1 \\
				\vdots & \vdots & \vdots & \vdots & \vdots & \ddots & \vdots & \vdots & \vdots \\
				1 & -1 & -1 & -1 & -1 & \cdots & 0 & 1 & 1 \\
				-1 & -1 & -1 & -1 & -1 & \cdots & -1 & 0 & 1 \\
				-1 & -1 & -1 & -1 & -1 & \cdots & -1 & -1 & 1 \\
				-1 & -1 & -1 & -1 & -1 & \cdots & -1 & -1 & 0 \\
				\end{array}\right|_{(n-1)\times(n-1)}} \notag\\
			&=(-2)\times\tiny{\left|\begin{array}{rrrrrrr|rr}
				-1 & 0 & 1 & 1 & 1 & \cdots & 1 & 1 & 1 \\
				1 & -1 & 0 & 1 & 1 & \cdots & 1 & 1 & 1 \\
				-1 & -1 & -1 & 0 & 1 & \cdots & 1 & 1 & 1 \\
				1 & -1 & -1 & -1 & 0 & \cdots & 1 & 1 & 1 \\
				\vdots & \vdots & \vdots & \vdots & \vdots & \ddots & \vdots & \vdots & \vdots \\
				1 & -1 & -1 & -1 & -1 & \cdots & 0 & 1 & 1 \\
				-1 & -1 & -1 & -1 & -1 & \cdots & -1 & 0 & 1 \\
				\cline{1-9}
				0 & 0 & 0 & 0 & 0 & \cdots & 0 & -1 & 0 \\
				0 & 0 & 0 & 0 & 0 & \cdots & 0 & 0 & -1 \\
				\end{array}\right|_{(n-1)\times(n-1)}} \notag\\
			&=(-2)\times\tiny{\left|\begin{array}{rrrrrrrr}
				-1 & 0 & 1 & 1 & 1 & \cdots & 1\\
				1 & -1 & 0 & 1 & 1 & \cdots & 1\\
				-1 & -1 & -1 & 0 & 1 & \cdots & 1\\
				1 & -1 & -1 & -1 & 0 & \cdots & 1\\
				\vdots & \vdots & \vdots & \vdots & \vdots & \ddots & \vdots\\
				(-1)^{n-4} & -1 & -1 & -1 & -1 & \cdots & 0\\
				(-1)^{n-3} & -1 & -1 & -1 & -1 & \cdots & -1\\
				\end{array}\right|_{(n-3)\times(n-3)}} \notag\\
			&=2\times\tiny{\left|\begin{array}{rrrrrrrr}
				1 & 0 & 1 & 1 & 1 & \cdots & 1\\
				-1 & -1 & 0 & 1 & 1 & \cdots & 1\\
				1 & -1 & -1 & 0 & 1 & \cdots & 1\\
				-1 & -1 & -1 & -1 & 0 & \cdots & 1\\
				\vdots & \vdots & \vdots & \vdots & \vdots & \ddots & \vdots\\
				(-1)^{n-5} & -1 & -1 & -1 & -1 & \cdots & 0\\
				(-1)^{n-4} & -1 & -1 & -1 & -1 & \cdots & -1\\
				\end{array}\right|_{(n-3)\times(n-3)}} \notag\\
			&=2Q_{n-3} \label{ak5}.
\end{align}

By using Theorem \ref{Djoin} and $\det\left(\begin{array}{rr}
0 & 1 \\
-1 & 0 \\
\end{array}\right)=1$, we have
\begin{align}
	\det(A_1)&=\tiny{\left|\begin{array}{rrrrrrr|rr}
		0 & 1 & -1 & 1 & -1 & \cdots & 1 & 1 & 1 \\
		-1 & 0 & 1 & 1 & 1 & \cdots & 1 & 1 & 1 \\
		1 & -1 & 0 & 1 & 1 & \cdots & 1 & 1 & 1 \\
		-1 & -1 & -1 & 0 & 1 & \cdots & 1 & 1 & 1 \\
		1 & -1 & -1 & -1 & 0 & \cdots & 1 & 1 & 1 \\
		\vdots & \vdots & \vdots & \vdots & \vdots & \ddots & \vdots & \vdots & \vdots \\
		-1 & -1 & -1 & -1 & -1 & \cdots & 0 & 1 & 1 \\
		\cline{1-9}
		-1 & -1 & -1 & -1 & -1 & \cdots & -1 & 0 & 1 \\
		-1 & -1 & -1 & -1 & -1 & \cdots & -1 & -1 & 0 \\
		\end{array}\right|}\notag\\
	&=\tiny{\left|\begin{array}{rrrrrrr}
			0 & 1 & -1 & 1 & -1 & \cdots & 1 \\
			-1 & 0 & 1 & 1 & 1 & \cdots & 1 \\
			1 & -1 & 0 & 1 & 1 & \cdots & 1 \\
			-1 & -1 & -1 & 0 & 1 & \cdots & 1 \\
			1 & -1 & -1 & -1 & 0 & \cdots & 1 \\
			\vdots & \vdots & \vdots & \vdots & \vdots & \ddots & \vdots \\
			-1 & -1 & -1 & -1 & -1 & \cdots & 0 \\
			\end{array}\right|}\notag\\
	&=\det(L_{n-2}). \label{ak6}
\end{align}

Since $n-2$ is even, by using Proposition \ref{Dettransi}, we have
\begin{align}
	\det(A_4)&=\tiny{\left|\begin{array}{rrrrrrrrr}
		0 & 0 & 0 & 0 & 0 & \cdots & 0 & \textbf{-2} & 0 \\
		0 & 0 & 1 & 1 & 1 & \cdots & 1 & 1 & 1 \\
		0 & -1 & 0 & 1 & 1 & \cdots & 1 & 1 & 1 \\
		0 & -1 & -1 & 0 & 1 & \cdots & 1 & 1 & 1 \\
		0 & -1 & -1 & -1 & 0 & \cdots & 1 & 1 & 1 \\
		\vdots & \vdots & \vdots & \vdots & \vdots & \ddots & \vdots & \vdots & \vdots \\
		0 & -1 & -1 & -1 & -1 & \cdots & 0 & 1 & 1 \\
		\textbf{2} & -1 & -1 & -1 & -1 & \cdots & -1 & 0 & 1 \\
		0 & -1 & -1 & -1 & -1 & \cdots & -1 & -1 & 0 \\
		\end{array}\right|}\notag\\
		&=\tiny{4\left|\begin{array}{rrrrrrr}
		0 & 1 & 1 & 1 & \cdots & 1 & 1 \\
		-1 & 0 & 1 & 1 & \cdots & 1 & 1 \\
		-1 & -1 & 0 & 1 & \cdots & 1 & 1 \\
		-1 & -1 & -1 & 0 & \cdots & 1 & 1 \\
		\vdots & \vdots & \vdots & \vdots & \ddots & \vdots & \vdots \\
		-1 & -1 & -1 & -1 & \cdots & 0 & 1 \\
		-1 & -1 & -1 & -1 & \cdots & -1 & 0 \\
		\end{array}\right|=4}.\label{ak7}
\end{align}

Combining (\ref{ak4}), (\ref{ak5}), (\ref{ak6}), (\ref{ak7}) and Lemma \ref{ledetLn}, we have
\begin{align*}
\det(L_n)&=\det(L_{n-2})+4Q_{n-3}+4=\det(L_{n-2})+4(n-2),
\end{align*}
thus (i) holds.

Note that $\det(L_4)=9$, by using (i) and induction, we have
\begin{align*}
	\det(L_n)=\det(L_4)+4(n-2)+4(n-4)+\cdots+4\cdot 4=(n-1)^2.
\end{align*}
Therefore, (ii) holds.
\end{proof}

By using Theorem \ref{detLn}, we obtain the following theorem, which solves Question \ref{queanyk}.

\begin{theorem}\label{AnyOdd}
	Let $k$ be odd. Then there exists a tournament whose determinant is $k^2$.
\end{theorem}

Let $n$ be a positive even integer, $T$ be an $n$-tournament with vertices $v_1,v_2,\ldots,v_n$, $X=\{v_1,v_2,\ldots,v_{n-1}\}$ such that $T[X]$ is transitive and $\psi_T(v_n,X)=(\alpha_1,\alpha_2,\ldots,\alpha_t)$. We will show that $\det(T)$ reaches its maximum if and only if $t=n-1$.

 \begin{proposition}\label{MaxLn}
Let $n$ be a positive even integer, $T$ be an $n$-tournament with vertices $v_1,v_2,\ldots,v_n$, $X=\{v_1,v_2,\ldots,v_{n-1}\}$ such that $T[X]$ is transitive and $\psi_T(v_n,X)=(\alpha_1,\alpha_2,\ldots,\alpha_t)$. Then $\det(T)\leq (n-1)^2$, equality holds if and only if $t=n-1$.
 \end{proposition}

\begin{proof}
	It is trivial when $n=2$. Assume that the result is true when $n\leq m$ with $m\geq 2$, now we prove that the result is also true when $n=m+2$. 
	
	Without loss of generality, we assume that $v_1\rightarrow v_2\rightarrow \cdots \rightarrow v_{n-1}$ in $T$ since $T[X]$ is transitive. Let $V(L_{t+1})=\{u_1,u_2,\ldots,u_t,u_{t+1}\}$ such that $u_1\rightarrow u_2\rightarrow \cdots \rightarrow u_t$ in $L_{t+1}$ and $\psi_{L_{t+1}}(u_{t+1},X_u)=((-1)^0,(-1)^1,\ldots,(-1)^{t-1})$, where $X_u=\{u_1,u_2,\ldots,u_{t}\}$.
	
	\textbf{Case 1}: $t\leq n-2$.
	
	It is clear that there exists $j$ such that $|\alpha_j|\geq2$. Let $W=\emptyset$ if $\alpha_1>0$, $W=\{v_n\}$ if $\alpha_1<0$, $X(i,\alpha_i)(i=1,\ldots,t)$ be the corresponding vertex sets in $\psi_T(v_n,X)=(\alpha_1,\alpha_2,\ldots,\alpha_t)$. Then $T$ is switching equivalent to $L_{t+1}(|\alpha_1|,|\alpha_2|,\ldots,|\alpha_t|,1)$ with respect to $W$, where $T[X(i,\alpha_i)]$ replaces $u_i(i=1,2,\ldots,t)$ and $T[v_n]$ replaces $u_{t+1}$. Then by using Proposition \ref{blowup}, we have $\det(T)=\det(L_{t+1}[U])$, where $U=\{u_i\in V(L_{t+1})\backslash\{u_{t+1}\}: |\alpha_i| \text{ is odd }\}\cup\{u_{t+1}\}$. Thus $1\leq |U|\leq t+1 \leq n-1$. It is clear that $L_{t+1}[U\backslash\{u_{t+1}\}]$ is transitive. If $|U|$ is odd, then $\det(T)=\det(L_{t+1}[U])=0<(n-1)^2$. If $|U|$ is even, then $|U|\leq n-2=(m+2)-2=m$ since $n$ is even. Therefore, by the assumption that the result is true when $n\leq m$, we have $\det(T)=\det(L_{t+1}[U])\leq (m-1)^2<(m+1)^2=(n-1)^2$ in this case. 

	\textbf{Case 2}: $t=n-1$.
	
	It is clear that $|\alpha_1|=\cdots=|\alpha_{n-1}|=1$. Let $W=\emptyset$ if $\alpha_1=1$, and $W=\{v_n\}$ if $\alpha_1=-1$. Then $T$ is switching equivalent to $L_n$ with respect to $W$, thus $\det(T)=\det(L_n)=(n-1)^2$ by Theorem \ref{detLn}. Therefore, $\det(T)=(n-1)^2$ if $t=n-1$.
	
	Combining the above two cases, the result is true when $n=m+2$. By induction, the result is true for any positive even integer $n$.
\end{proof}

\begin{proposition}\label{SubLn}
	Let $n$ be a positive even integer. Then for every subtournament $H$ of $L_n$, $\det(H)\leq\det(L_n)=(n-1)^2$, and equality holds if and only if $H$ is $L_n$.
\end{proposition}

\begin{proof}
	If $|V(H)|=1$, then $\det(H)=0<(n-1)^2$. Now we consider $|V(H)|=m>1$.
	
	Let $V(L_n)=\{u_1,\ldots,u_{n-1},u_n\}$ such that $u_1\rightarrow \cdots \rightarrow u_{n-1}$ and $\psi_{L_n}(u_n,X_u)=((-1)^0,\ldots,(-1)^{n-2})$, where $X_u=\{u_1,\ldots,u_{n-1}\}$.
	
	\textbf{Case 1}: $u_n\notin V(H)$.
	
	It is clear that $H$ is transitive, thus $\det(H)\leq 1\leq (n-1)^2$. 
	
	\textbf{Case 2}: $u_n\in V(H)$.
	
	It is clear that $H(V(H)\backslash\{u_n\})$ is transitive and there exist nonzero integers $\alpha_1,\ldots,\alpha_t$ such that $\psi_H(u_n,V(H)\backslash\{u_n\})=(\alpha_1,\ldots,\alpha_t)$. Therefore, by Proposition \ref{MaxLn}, $\det(H)\leq (m-1)^2\leq (n-1)^2$, equality holds if and only if $m=n$, or equivalently, $|V(H)|=|V(L_n)|$, which implies $H$ is $L_n$.
\end{proof}

Now we give the following theorem, it is a stronger version of Theorem \ref{AnyOdd}.

\begin{theorem}\label{AnyOddsub}
	Let $k$ be odd. Then there exists a tournament $T$ such that $M_T=\{0,1^2,\ldots,k^2\}$, where $M_T$ is the set consisting of all the determinants of subtournaments of $T$.
\end{theorem}

\begin{proof}
	Note that $L_{k+1}$ contains $L_2,L_4,\ldots,L_{k-1},L_{k+1}$ as its subtournaments. Therefore, for every $m^2\in\{1^2,3^2,\ldots,(k-2)^2,k^2\}$, there exists a subtournament of $L_{k+1}$ whose determinant is $m^2$. Moreover, the determinant of a tournament with order $1$ is $0$, thus $0\in M_{L_{k+1}}$. Then by Proposition \ref{SubLn}, we have $M_{L_{k+1}}=\{0,1^2,3^2,\ldots,k^2\}$. 
	
	Let $T$ be $L_{k+1}$. Then we complete the proof.
\end{proof}

By the proof of Theorem \ref{AnyOddsub}, we have $L_{k+1}\in \mathcal{D}_k\backslash\mathcal{D}_{k-2}$. Therefore, $\mathcal{D}_k\backslash\mathcal{D}_{k-2}$ is not an empty set.

\begin{theorem}\label{Notempty}
	For any odd integer $k\geq 3$, $\mathcal{D}_k\backslash\mathcal{D}_{k-2}\neq\emptyset$.
\end{theorem}

Now Question \ref{queanyDk} can be expressed as follows.

\begin{question}\label{queanyDktwo}
	Let $k\text{ }(\geq 7)$ be odd. What conditions must $T$ and $L_{k+1}$ satisfy such that $T\in \mathcal{D}_k\backslash \mathcal{D}_{k-2}$ if and only if $T$ is switching equivalent to a transitive blowup of $L_{k+1}$?
\end{question}

\begin{remark}\label{reNotempty}
	Since $L_{k+1}\in \mathcal{D}_k\backslash\mathcal{D}_{k-2}$, where $k\geq 3$ and $k$ is odd, then for any sequence of positive integers $a_1,\ldots,a_{k+1}$, $L_{k+1}(a_1,\ldots,a_{k+1})\in \mathcal{D}_k\backslash\mathcal{D}_{k-2}$ by Proposition \ref{blowup} and Remark \ref{reblowupclass}. Therefore, there exists an $n$-tournament $T$ such that $T\in \mathcal{D}_k\backslash\mathcal{D}_{k-2}$ for any integer $n\geq k+1$.
\end{remark}

\vspace{0.5em}

\noindent
{\bf Funding}\,

This work is supported by the National Natural Science Foundation of China (Grant Nos.12371347, 12271337).

\noindent
{\bf Acknowledgments}\,

This work is supported by the National Natural Science Foundation of China (Grant Nos.12371347, 12271337).

\noindent
{\bf Conflict of interest statement}\,

The authors declare no conflict of interest.

\end{document}